\documentclass{article}
\usepackage[utf8]{inputenc}

\title{Colimits of categories, zig-zags and necklaces }
\author{Redi Haderi \thanks{Throughout the preparation of this document, the author has been supported, financially and otherwise, by his wife Naisila Puka.}}

\date{\today}

\usepackage[english]{babel}
\usepackage[utf8]{inputenc}
\usepackage{amsmath}
\usepackage{graphicx}
\usepackage{quiver}
\usepackage{mathtools}
\usepackage{bbold}

\usepackage{url}

\usepackage{amsthm}

\usepackage{hyperref}
\usepackage{amssymb}

\usepackage{tikz}
\usetikzlibrary{arrows}
\usetikzlibrary{cd}

\usepackage{array}
\usepackage{epigraph}

\newtheorem{theorem}{Theorem}[section]
\newtheorem{proposition}{Proposition}[section]
\newtheorem{definition}{Definition}[section]

\newtheorem{lemma}{Lemma}[section]

\theoremstyle{remark}

\theoremstyle{remark}
\newtheorem*{remark}{Remark}

\theoremstyle{remark}
\newtheorem*{note}{Note}

\theoremstyle{remark}
\newtheorem{construction}{Construction}[section]

\makeatletter
\def\slashedarrowfill@#1#2#3#4#5{%
  $\m@th\thickmuskip0mu\medmuskip\thickmuskip\thinmuskip\thickmuskip
   \relax#5#1\mkern-7mu%
   \cleaders\hbox{$#5\mkern-2mu#2\mkern-2mu$}\hfill
   \mathclap{#3}\mathclap{#2}%
   \cleaders\hbox{$#5\mkern-2mu#2\mkern-2mu$}\hfill
   \mkern-7mu#4$%
}
\def\rightslashedarrowfill@{%
  \slashedarrowfill@\relbar\relbar\mapstochar\rightarrow}
\newcommand\xslashedrightarrow[2][]{%
  \ext@arrow 0055{\rightslashedarrowfill@}{#1}{#2}}
\makeatother

\DeclareMathOperator*{\colim}{colim}

\newcommand{\cc}[1]{\mathcal{#1}}
\newcommand{\bb}[1]{\mathbb{#1}}
\newcommand{\ff}[1]{\mathfrak{#1}}

\newcommand{\ra}[1]{\overrightarrow{#1}}

\newcommand{\D}[1]{\Delta^{#1}}
\newcommand{\set}{\textbf{Set}}

\newcommand{\sset}{\textbf{sSet}}

\newcommand{\cat}{\textbf{Cat}}
\newcommand{\dcat}{\textbf{DCat}}

\newcommand{\spx}{\cc{S}(X)}

\usepackage{biblatex}

\begin{document}

\maketitle

\begin{abstract}
    Given a diagram of small categories $F : \cc{J} \rightarrow \cat$, we provide a combinatorial description of its colimit in terms of the indexing category $\cc{J}$ and the categories and functors in the diagram $F$. We introduce certain double categories of zig-zags in order to keep track of the necessary identifications. We found these double categories necessary, but also explanatory. 
    
    When applied pointwise in the simplicially enriched setting, our constructions offer a shorter proof of the necklace theorem of Dugger and Spivak by direct computation. 
\end{abstract}

\section{Introduction and summary} \label{sec_intro}

This note was originally intended as an appendix in another project, but grew as an independent work. The motivating objects of study are small diagrams 
$$F : \cc{J} \rightarrow \cat_\D{}$$
in the category of simplicially enriched categories. 
The colimits of such diagrams are generally deemed hard to understand. Important constructions, such as the free simplicially enriched category $\ff{C} X$ generated by a simplicial set $X$, arise as such colimits. 

\subsection*{Colimits of simplicial categories}

Let $\cc{C} \in \cat_{\D{}}$. Given objects $a, b \in \cc{C}$ and $[p] \in \D{}$, we refer to a $p$-simplex $u \in \cc{C}(a,b)_p$ in the mapping space as a $p$-\textit{arrow}. Let $\cc{C}_p$ be the category whose objects are those of $\cc{C}$ and morphisms are $p$-arrows.

The assignment $[p] \mapsto \cc{C}_p$ creates a simplicial object in $\cat$
$$\cc{C}_* : \D{op} \rightarrow \cat$$
which, in turn, produces a functor
$$(-)_* : \cat_{\D{}} \rightarrow \cat^{\D{op}}$$
The functor $(-)_*$ is a fully faithful embedding. Its image consists of \textit{discrete} simplicial objects in $\cat$ : those objects $\cc{D} \in \cat^{\D{op}}$ for which the categories $\cc{D}_p$ have the same objects for all $[p]$, and such that the functors $\theta^* : \cc{D}_p \rightarrow \cc{D}_q$ are identity on objects for all morphisms $\theta : [q] \rightarrow [p]$ in $\D{}$. 

The functor $(-)_*$ admits a right adjoint
$$(-)_* : \cat_{\D{}} \leftrightarrows \cat^{\D{op}} : (-)_v$$
The functor $(-)_v$ assigns to a simplicial object $\cc{D}$ its \textit{vertical}\footnote{This term is adopted from double category theory, in an analogy developed in \cite{haderi2023simplicial}.} simplicial category $\cc{D}_v$. Roughly speaking, $\cc{D}_v$ is the subobject of $\cc{D}$ spanned by totally degenerate objects of $\cc{D}_p$ for $[p] \in \D{}$, or equivalently the largest discrete subobject of $\cc{D}$. 

In light of the above, given a diagram 
$$F : \cc{J} \rightarrow \cat_{\D{}}$$
of simplicial categories we will have 
$$(\colim_{i \in \cc{J}} F(i))_p = \colim_{i \in \cc{J}} F(i)_p$$
for all $[p] \in \D{}$. Hence, the problem of understanding the left-hand side reduces to understanding the right-hand side, which is a colimit in $\cat$.

\subsection*{Colimits of categories}

We turn our attention to (small) diagrams of (small) categories 
$$F : \cc{J} \rightarrow \cat$$
For practical ease, we adopt the following notation for the rest of this paper:
\[
\begin{matrix}
    F: & i & \mapsto & \cc{C}_i \\
    & u : i \rightarrow j & \mapsto & \tilde{u} : \cc{C}_i \rightarrow \cc{C}_j
\end{matrix}
\]
Let $\cc{C} = \colim\limits_{i \in \cc{J}} \cc{C}_i$ 

In some sense, there is no problem in constructing the colimit $\cc{C}$. For example, utilizing the free category-nerve adjunction $\tau : \sset \leftrightarrows \cat : N$, we may derive that $\cc{C} \cong \tau(\colim\limits_{i \in \cc{J}} N(\cc{C}_i))$ (for instance, see \cite[Corollary 4.5.16]{riehl2017category}). 

While we may use the above observation to imply the existence of $\cc{C}$, its structure remains \textit{implicit}. On objects, given that the functor $Ob : \cat \rightarrow \set$, which assigns to a small category its set of objects, has a right adjoint, we may conclude that 
$$Ob(\cc{C}) \cong \colim_{i \in \cc{J}} Ob(C_i)$$
The right-hand side is well-understood as a colimit in $\set$. This way, the objects of $\cc{C}$ are equivalence classes $[a]$ represented by objects $a \in \cc{C}_i$, $i \in \cc{J}$.

The problem is in describing the morphism sets $\cc{C}([a], [b])$ given two objects $[a], [b] \in \cc{C}$. An explicit description would have to be in terms on the objects and morphisms of the categories $\cc{C}_i$, $i \in \cc{J}$, and the diagram $F$.

Taking a step back, observe that, in constructing $\cc{C}$, for a morphism $u : i \rightarrow j$ in $\cc{J}$ and a morphism $f: a \rightarrow b$ in $\cc{C}_i$, we ought to identify $a \sim \Tilde{u}(a)$, $b \sim \Tilde{u}(b)$ and $f \sim \Tilde{u}(f)$. 
These identifications create the following basic type of situation. In case of a \textit{roof} in $\cc{J}$,
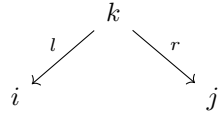
\[\begin{tikzcd}
	& k \\
	i && j
	\arrow["l"', from=1-2, to=2-1]
	\arrow["r", from=1-2, to=2-3]
\end{tikzcd}\]
consider objects $a \in \cc{C}_i$, $b \in \cc{C}_j$, $c \in \cc{C}_k$ and morphisms $f : a \rightarrow \Tilde{l}(c)$, $g : \Tilde{r}(c) \rightarrow b$
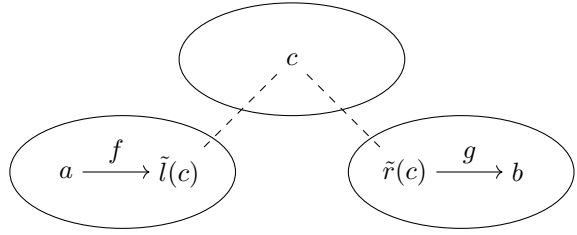
\begin{center}
    \begin{tikzpicture}[baseline=(current bounding box.center)]
 \draw[thin] (0.75,0) ellipse (1.5cm and 0.75cm);
 \draw[thin] (5.25,0) ellipse (1.5cm and 0.75cm);
  \draw[thin] (3,1.5) ellipse (1.5cm and 0.75cm);
  
  \node (a) at (0,0) {$a$};
  \node (tildec) at (1.5,0) {$\tilde{l}(c)$};
  \node (tilder) at (4.5,0) {$\tilde{r}(c)$};
  \node (b) at (6,0) {$b$};
  \node (c) at (3,1.5) {$c$};

  \draw[->] (a) -- node[above] {$f$} (tildec);
  \draw[->] (tilder) -- node[above] {$g$} (b);
  \draw[dashed] (c) -- (tildec);
  \draw[dashed] (c) -- (tilder);
\end{tikzpicture}
\end{center}
After the identifications, the arrows $f$ and $g$ become composable in $\cc{C}$ and create a new morphism $(f,g)$. 

General morphisms in $\cc{C}$ are simply such sequences, except possibly longer and indexed by \textit{zig-zags} in $\cc{J}$ rather than mere roofs. This idea is formulated in the following theorem, which is our first contribution.

\begin{theorem}
    With notation as above, an element of the set $\cc{C}([a], [b])$ is represented by the following three pieces of data:
    \begin{itemize}
        \item [(i)] A choice of representatives $a \in C_i$ and $b \in C_j$ for some $i,j \in \cc{J}$.
        
        \item[(ii)] A zig-zag in $\cc{J}$ which connects $i$ and $j$
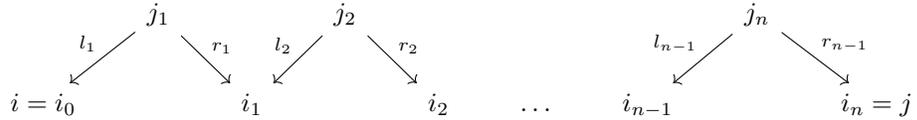
        \[\begin{tikzcd}[column sep=scriptsize]
	& {j_1} && {j_2} &&&& {j_n} \\
	{i=i_0} && {i_1} && {i_2} & \dots & {i_{n-1}} && {i_n = j}
	\arrow["{l_1}"', from=1-2, to=2-1]
	\arrow["{l_2}"', from=1-4, to=2-3]
	\arrow["{l_{n-1}}"', from=1-8, to=2-7]
	\arrow["{r_1}", from=1-2, to=2-3]
	\arrow["{r_{n-1}}", from=1-8, to=2-9]
	\arrow["{r_2}", from=1-4, to=2-5]
\end{tikzcd}\]

    \item[(iii)] A sequence of objects $a_k \in \cc{C}_{j_k}$ and a  “chain" of morphisms $f_k \in C_{i_k}$ of the following form:
    $$a \xrightarrow{f_0} \Tilde{l}_1(a_1) \  , \  \Tilde{r}_1(a_1) \xrightarrow{f_1} \Tilde{l}_2(a_2) \ , \dots , \  \Tilde{r}_n(a_n) \xrightarrow{f_n} b$$
    \end{itemize}
     subject to a relation described in Section \ref{sec_colimthm}. 
\end{theorem}

After our thought experiment, the data postulated in the above theorem is intuitive and not difficult to organize and keep track of. For instance, we regard zig-zags as in (ii) as a notion of morphism on the objects of $\cc{J}$, depicted as
\[\begin{tikzcd}
	i & j
	\arrow["{(\bar{l}, \bar{r})}", squiggly, from=1-1, to=1-2]
\end{tikzcd}\]
Composing by concatenation, we form the category of zig-zags $\cc{Z} \cc{J}$. 

We refer to a chain of morphisms as in (iii) as an $F$-\textit{decoration} of the zig-zag $(\Bar{l}, \Bar{r})$. Such a decorated zig-zag can also be thought of as a morphism
\[\begin{tikzcd}
	{(i,a)} & {(j,b)}
	\arrow["{(\bar{l}, \bar{r}, \bar{f})}", squiggly, from=1-1, to=1-2]
\end{tikzcd}\]
in a category $\cc{Z} F$ whose set of objects is $\bigoplus_{i \in \cc{J}} Ob(\cc{C}_i)$.

To address the challenging bit of this line of thought, which is keeping track of the necessary relations, we resort to an excellent organizing structure: \textit{double categories}. We construct double categories $\bb{Z} \cc{J}$ and $\bb{Z} F$ whose 2-cells
\[\begin{tikzcd}
	i & j \\
	{i^\prime} & {j^\prime}
	\arrow[""{name=0, anchor=center, inner sep=0}, squiggly, from=1-1, to=1-2]
	\arrow[""{name=1, anchor=center, inner sep=0}, squiggly, from=2-1, to=2-2]
	\arrow[from=1-1, to=2-1]
	\arrow[from=1-2, to=2-2]
	\arrow["\rho", shorten <=4pt, shorten >=4pt, Rightarrow, from=0, to=1]
\end{tikzcd} \ \ \, \ \ \ 
\begin{tikzcd}
	{(i,a)} & {(j,b)} \\
	{(i^\prime, a^\prime)} & {(j^\prime, b^\prime)}
	\arrow[""{name=0, anchor=center, inner sep=0}, squiggly, from=1-1, to=1-2]
	\arrow[""{name=1, anchor=center, inner sep=0}, squiggly, from=2-1, to=2-2]
	\arrow[from=1-1, to=2-1]
	\arrow[from=1-2, to=2-2]
	\arrow["{\rho^*}", shorten <=4pt, shorten >=4pt, Rightarrow, from=0, to=1]
\end{tikzcd}
\]
encode the appropriate notion of transformation between zig-zags and their decorations. These double categories will be introduced in Section \ref{sec_dzig}.

\subsection*{Reinterpreting the necklace theorem}

Let $X$ be a simplicial set. It is of interest in higher category theory to understand the free simplicial category $\ff{C}X$ generated by $X$. The functor 
$$\ff{C} : \sset \rightarrow \cat_\D{}$$
is the left adjoint of the coherent nerve functor, and this pair has been shown to be a Quillen equivalence between the Joyal model structure on $\sset$ and the Bergner model structure on $\cat_\D{}$ (\cite{bergner2007model}) by Lurie (\cite{lurie2009higher}). In case $X$ is an $\infty$-category (i.e. quasi-category), the simplicial category $\ff{C} X$ may be interpreted as a strictification of some of the compositions in $X$. In case $X$ is (the nerve of) a category, $\ff{C} X$ provides a cofibrant replacement in the Bergner model structure ($X$ being regarded as simplicially enriched by discrete simplicial sets).

In the case $X = \D{n}$, we denote $\ff{C}\D{n} = \bb{\Delta}^n$. For $[p] \in \D{}$, the category of $p$-arrows $\bb{\D{n}}_p$ may be combinatorially described as follows:
\begin{itemize}
    \item [$(\bullet)$] The objects are $0, 1, \dots , n$.
    \item [$(\rightarrow)$] For two objects $i, j$, a $p$-arrow $\ra{U} : i \rightarrow j$ is a \textit{flag} 
    $$U^0 \subseteq U^1 \subseteq \dots \subseteq U^p$$
    of subsets of $\{ i, \dots , j \}$ such that $i, j \in U^0$.
    \item [$(\circ)$] Composition is given by union of sets.
\end{itemize}
For a morphism $\theta : [q] \rightarrow [p]$ in $\D{}$, the identity on objects functor $\theta^* : \bb{\D{n}}_p \rightarrow \bb{\D{n}}_q$ is given, for a morphism $\ra{U} : i \rightarrow j$, by
$$\theta^*(\ra{U}) = (U^{\theta(0)} \subseteq \dots \subseteq U^{\theta(p)})$$
In particular, faces and degeneracies are obtained by deletion and repetition of terms in the flag $\ra{U}$.

We extend the free simplicial category functor $\ff{C}$ from $\D{n}$ to general simplicial sets via Yoneda extension. This means the following. For a simplicial set $X$, let $\spx$ denote its category of simplices, given by:
\begin{itemize}
    \item [$(\bullet)$] Objects the simplices $x \in X_n$ for $[n] \in \D{}$.
    \item[$(\rightarrow)$] Morphisms $x \rightarrow y$ commutative triangles of simplicial maps
    \[\begin{tikzcd}
	{\D{n}} && {\D{n}} \\
	& X
	\arrow["x"', from=1-1, to=2-2]
	\arrow["y", from=1-3, to=2-2]
	\arrow["\theta", from=1-1, to=1-3]
\end{tikzcd}\]
\end{itemize}
The assignment $(x \in X_n) \mapsto \bb{\D{n}}$ gives rise to a diagram of simplicial categories
$$\chi : \spx \rightarrow \cat_{\D{}}$$
We define $\ff{C} X$ to be the colimit of this functor, typically expressed as
$$\ff{C}X \cong \colim_{x \in X_n} \bb{\Delta}^n$$

The objects of $\ff{C}X$ are the $0$-simplices of $X$. Given $a,b \in X_0$ and $[p] \in \D{}$, a combinatorial description of the $p$-arrows in $\ff{C} X (a,b)_p$, which extends the one presented above for the case $X = \D{n}$, has been provided by Dugger and Spivak  in \cite{dugger2011rigidification}. We briefly summarize.

Let $\sset_{*,*}$ be the category of bipointed simplicial sets. We regard $\D{n} \in \sset_{*,*}$ with chosen basepoints $0$ and $n$. The wedge sum  $\D{n} \vee \D{m}$ is provided by stringing $\D{n}$ and $\D{m}$ along the endpoint $n$ of $\D{n}$ and $0$ of $\D{m}$. 
A \textit{necklace} is a simplicial set of the form
$$N = \D{n_0} \vee \dots \vee \D{n_k}$$

For a necklace $N$ as above define the following:
\begin{itemize}
    \item [-] (Basepoints). $N$ is regarded as bipointed by the initial vertex of $\D{n_0}$ and the terminal vertex of $\D{n_k}$.
    \item[-] (Beads). The beads of the necklace $N$ are the simplicial sets $\D{n_i}$, $0 \leq i \leq k$. 
    \item[-] (Joins). The set $J_N$ of joins of the necklace $N$ is the union of basepoints of the beads (with the identifications provided by wedging).
    \item[-] (Vertices). The set $V_N$ of vertices of the necklace $N$ is the union of the vertices of the beads (with the identifications provided by wedging). 
\end{itemize}

\begin{theorem}[Dugger and Spivak, Theorem 1.4 \cite{dugger2011rigidification}]

Let $X$ be a simplicial set, $a, b \in X_0$ and $[p] \in \D{}$. An element of the $p$-arrow set $\ff{C}X(a,b)_p$ is represented by a pair $(N \rightarrow X_{a,b}, \ra{U})$ where
\begin{itemize}
    \item [(i)] $N$ is a necklace and $N \rightarrow X_{a,b}$ is a map of bipointed simplicial sets.
    \item[(ii)] $\ra{U}$ is a flag $$ U^0 \subseteq U^1 \subseteq \dots \subseteq U^p$$ of subsets of $V_N$ such that $J_N \subseteq U^0$.  
\end{itemize}
Two such pairs $(N \rightarrow X_{a,b}, \ra{U})$ and $(M \rightarrow X_{a,b}, \ra{V})$ are identified in case there is a commutative triangle of bipointed maps
\[\begin{tikzcd}
	N && M \\
	& {X_{a,b}}
	\arrow[from=1-1, to=2-2]
	\arrow[from=1-3, to=2-2]
	\arrow["f", from=1-1, to=1-3]
\end{tikzcd}\]
such that $f(\ra{U}) = \ra{V}$.

\end{theorem}

In light of our previous observations, for $[p] \in \D{}$, the category of $p$-arrows $(\ff{C}X)_p$ arises as the colimit of the diagram 
\[
\begin{matrix}
    \chi_p : & \spx & \rightarrow & \cat \\
    & x \in X_n & \mapsto & \bb{\D{n}}_p
\end{matrix}
\]
Therefore, we can implement our colimit theorem and understand the necklace theorem as follows: 

\begin{itemize}
    \item [(A)] The packet of data $(N \rightarrow X, \ra{U})$ representing a morphism $a \rightarrow b$ in $\ff{C} X(a,b)_p$ can be interpreted as:
    \begin{itemize}
        \item [-] A choice of simplices $x$ and $y$ in $X$ such that $a$ is the initial vertex of $x$ and $b$ is the terminal vertex of $y$.
        \item[-] A zig-zag $x \rightsquigarrow y$ in $\spx$ of the form
       \[\begin{tikzcd}[column sep=scriptsize]
	& {a_1} && \dots && {a_k} \\
	{x = x_0} && {x_1} & \dots & {x_{k-1}} && {x_k = y}
	\arrow["{\omega_1}"', from=1-2, to=2-1]
	\arrow["{\alpha_1}", from=1-2, to=2-3]
	\arrow["{\omega_n}"', from=1-6, to=2-5]
	\arrow["{\alpha_n}", from=1-6, to=2-7]
\end{tikzcd}\]
where $a_i$'s are vertices of $X$ and $\alpha$ and $\omega$ denote initial and terminal vertex inclusions into simplices. 
        \item[-] A chain of $p$-arrows $\ra{U}_1 : a \rightarrow a_1, \dots , \ra{U}_k : a_{n-1} \rightarrow b$, obtained by splitting the sets in $\ra{U}$ according to the joins of $N$. 
    \end{itemize}    

    \item[(B)] On the other hand, our colimit theorem states that a $p$-arrow $a \rightarrow b$ in $\ff{C}X(a,b)_p$ is, in general, represented by
    \begin{itemize}
        \item [-] A choice of simplices $x, y \in X$ such that $a$ is a vertex of $x$ and $b$ a vertex of $y$.

        \item[-] A connecting zig-zag between $x \rightsquigarrow y$ in $\spx$, say 
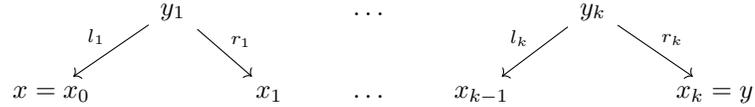
\[\begin{tikzcd}[column sep=scriptsize]
	& {y_1} && \dots && {y_k} \\
	{x = x_0} && {x_1} & \dots & {x_{k-1}} && {x_k = y}
	\arrow["{l_1}"', from=1-2, to=2-1]
	\arrow["{r_1}", from=1-2, to=2-3]
	\arrow["{l_k}"', from=1-6, to=2-5]
	\arrow["{r_k}", from=1-6, to=2-7]
\end{tikzcd}\]

        \item[-] A choice of vertices $a_i$ in $ x_i$ and a chain $\ra{U}_1 : a \rightarrow a_1 , \ra{U}_2 : a_1 \rightarrow a_2 , \dots , \ra{U}_n : a_n \rightarrow b$, where $\ra{U}_i$ is a $p$-arrow in the category $\bb{\D{n_i}}_p$ corresponding to the simplex $x_i : \D{n_i} \rightarrow X$.
    \end{itemize}

    \item[(C)] Upon scrutiny,  data as in (B) can always be replaced with data as in (A). For instance, in forming the $\chi_p$-decoration in (B) we are not making use of the simplices $y_i$ except in choosing the vertices $a_i$. Such a replacement of data can be carried out systematically, i.e. functorially. We present this in Construction \ref{const_ncsrep} and label it \textit{necklace replacement}. 

    \item[(D)] In the case of data as in (A), the relation between decorated zig-zags used in the colimit theorem translates, on the nose, to the relation in the necklace theorem. This is part of the content of Proposition \ref{prop_translate}. 
\end{itemize}

We could consider the above steps as a proof of the necklace theorem by \textit{direct computation}.
We find our approach to have three advantages:
\begin{itemize}
    \item [-] The proofs, in our opinion, seem to be more transparent as a result of the explicit nature of our colimit theorem.

    \item[-] We provide an explanation of why necklaces appear in the first place in a combinatorial description of $\ff{C}X$. They appear as a natural simplification of general zig-zags which appear in colimits of categories. Another way to put this is to say that necklaces in $X$ are a subclass of zig-zags in $\spx$ sufficient to record all $\chi_p$-decorations in forming the colimit $\ff{C}X$.

    \item[-] We provide an explanation of why the relation between the data in the necklace theorem is what it is. It is a translation in the language of simplicial sets of relations between zig-zags which appear in studying colimits of categories. Moreover, the latter are natural and sensible in their own right. 
\end{itemize}
Ultimately, we let the reader be the judge. This paper represents our attempt in understanding some of the concepts involved in these results.

\subsection*{Conventions}

All categories which appear in this paper are assumed to be small. We have assumed some familiarity with simplicial sets, and used fairly standard notation. Background on simplicial sets is only necessary for the parts of the paper dedicated to necklaces though. 

\subsection*{Note}

As far as we know, our colimit theorem for small categories does not appear in the literature in the form presented here, nor do our double categories of zig-zags or our proof of the necklace theorem.

We apologize in advance for any missing reference, and would be more than happy to retract any claim of originality and include relevant references in future versions of this paper. 

\section{Zig-zags and their origins} \label{sec_set}

Let 
\[
\begin{matrix}
    F : & \cc{J} & \rightarrow & \set \\
    & i & \mapsto & C_i \\
    & i \xrightarrow{u} j & \mapsto & C_i \xrightarrow{\Tilde{u}} C_j
\end{matrix}
\]
be a diagram of sets. An easy exercise in category theory reveals that the colimit of the diagram $F$ is the quotient of the set 
\begin{align*}
    DC &= \bigoplus_{i \in \cc{J}} C_i \\
    & = \{(i,a) \ : \ i \in \cc{J}, a \in \cc{C}_i \}
\end{align*}
by the following relation:
\begin{itemize}
    \item [$(E)$] A pair of elements $(i,a)$ and $(j,b)$ are in the relation $E$ if there is a morphism $u : i \rightarrow j$ in $\cc{J}$ such that $\Tilde{u}(a) = b$.
\end{itemize}
We would like to emphasize a couple of ideas which are implicitly present in the formation of colimits of sets. 

The relation $E$ is naturally recorded by a notion of morphism on the data set $DC$, by regarding a witness $u$ in the relation as a morphism $(i,a) \rightarrow (j,b)$. This is the well-known \textit{category of elements} $\cc{E}(F)$ associated to $F$. In fact, we have 
$$C \cong \pi_0 \cc{E}(F)$$
where $\pi_0 : \cat \rightarrow \set$ is the connected components functor.

The second type structure brought to attention by colimits of sets, $\pi_0$ being involved, is \textit{zig-zags}. In general, two objects of a category $\cc{A}$ represent the same class in $\pi_0 \cc{A}$ if and only if there is a zig-zag of morphisms connecting them. We record a little bit of zig-zag formalism.

 For $n \geq 0$, let $\cc{Z}^n$ be the walking zig-zag of the form
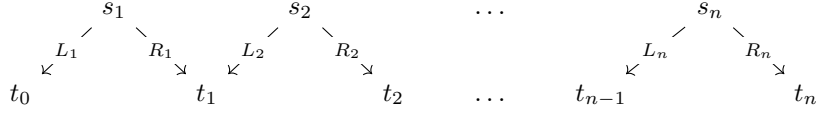
\[\begin{tikzcd}[column sep=scriptsize]
	& {s_1} && {s_2} && \dots && {s_n} \\
	{t_0} && {t_1} && {t_2} & \dots & {t_{n-1}} && {t_n}
	\arrow["{L_1}"{description}, from=1-2, to=2-1]
	\arrow["{R_1}"{description}, from=1-2, to=2-3]
	\arrow["{L_2}"{description}, from=1-4, to=2-3]
	\arrow["{R_2}"{description}, from=1-4, to=2-5]
	\arrow["{L_n}"{description}, from=1-8, to=2-7]
	\arrow["{R_n}"{description}, from=1-8, to=2-9]
\end{tikzcd}\]
obtained by stringing $n$ roofs along endpoints. 
The sequences of morphisms $\Bar{R}$ and $\Bar{L}$ are the morphisms pointing to the left and right respectively. In case $n = 0$, we let $\cc{Z}^0 = 1$ be the terminal category.

For a category $\cc{J}$, we define a zig-zag in $\cc{J}$ to be a functor 
$$(\Bar{l}, \Bar{r}) : \cc{Z}^n \rightarrow \cc{J}$$
The notation $(\Bar{l}, \Bar{r})$ indicates that the zig-zag is comprised of left pointing arrows $l_1, \dots , l_n$ and right pointing arrows $r_1, \dots , r_n$. In case $n = 0$, a functor $\cc{Z}^0 \rightarrow \cc{J}$ picks an object $i \in \cc{J}$. We denote this zig-zag $\{ i \}$ and refer to it as the \textit{trivial} zig-zag at $i$. 

We say that a zig-zag $\cc{Z}^n \rightarrow \cc{J}$ \textit{connects} two objects $i$ and $j$ in case $t_0 \mapsto i$ and $t_n \mapsto j$. For emphasis, we write $\cc{Z}^n \rightarrow \cc{J}_{i,j}$. In fact, we regard $\cc{Z}^n$ as a category with two chosen basepoints $t_0$ and $t_n$. The category of small categories with two chosen basepoints and functors which preserves them is denoted $\cat_{*,*}$. This way $\cc{J}_{i,j} \in \cat_{*,*}$ with basepoints $i$ and $j$ and zig-zags connecting $i$ and $j$ are functors in $\cat_{*,*}$. 
We depict such zig-zags with curly arrows
\[\begin{tikzcd}
	i & j
	\arrow["{(\Bar{l}, \Bar{r})}", squiggly, from=1-1, to=1-2]
\end{tikzcd}\]

Being bipointed categories, we have a wedging operation on walking zig-zags (which strings two zig-zags together), so that 
$$\cc{Z}^n \vee \cc{Z}^m \cong \cc{Z}^{n+m}$$
This operation induces the \textit{concatenation} operation on zig-zags in $\cc{J}$. We obtain a category, denoted $\cc{ZJ}$, whose objects are those of $\cc{J}$ and morphisms are zig-zags.

\begin{definition}[$F$-decoration]
    Let $F : \cc{J} \rightarrow \set$ be a diagram and $(\Bar{l}, \Bar{r}) : i \rightsquigarrow j$ be a zig-zag, say
    
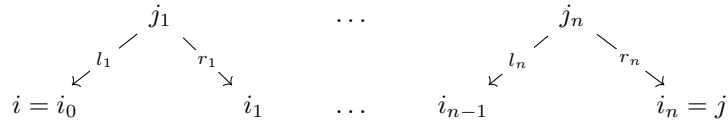
\[\begin{tikzcd}[column sep=scriptsize]
	& {j_1} && \dots && {j_n} \\
	{i = i_0} && {i_1} & \dots & {i_{n-1}} && {i_n = j}
	\arrow["{l_1}"{description}, from=1-2, to=2-1]
	\arrow["{r_1}"{description}, from=1-2, to=2-3]
	\arrow["{l_n}"{description}, from=1-6, to=2-5]
	\arrow["{r_n}"{description}, from=1-6, to=2-7]
\end{tikzcd}\]
In notation as above, an $F$-decoration of $(\Bar{l}, \Bar{r})$ consists of a sequence $\Bar{a} = (a_1, \dots , a_{n-1})$, where $a_k \in C_{j_k}$ and such that $\Tilde{r}_k(a_k) = \Tilde{l}_{k+1}(a_{k+1})$ for all $k$.
\end{definition}

Given two elements $(i,a), (j,b) \in DC$, we write
\[\begin{tikzcd}[column sep=scriptsize]
	{(i,a)} && {(j,b)}
	\arrow["{(\Bar{l}, \Bar{r}, \Bar{a})}", squiggly, from=1-1, to=1-3]
\end{tikzcd}\]
to indicate an $F$-decorated zig-zag as in the above definition which connects $a$ and $b$, i.e. such that $\Tilde{l}_1(a_1) = a$ and $\Tilde{r}_n(a_n) = b$. By concatenating in the evident manner, we obtain a category $\cc{Z} F$ with set of objects $DC$. 

To obtain the colimit of $F$, we may form the quotient of $DC$ by the following relation:
\begin{itemize}
    \item [$(Z)$] A pair of elements $(i,a)$ and $(j,b)$ are in the relation $Z$ if there is an $F$-decorated zig-zag $(i,a) \rightsquigarrow (j,b)$ connecting them.
\end{itemize}
With respect to $E$, the relation $Z$ has the advantage of being an equivalence relation. Moreover, it is more explicit and, arguably, explanatory.

\begin{remark}
    The category $\cc{Z} F$ defined above is just the category of zig-zags $\cc{Z}\cc{E}(F)$ associated to the category of elements. While we could have defined $\cc{Z} F$ this way, we intended to emphasize the notion of $F$-decoration, a categorified version of which will be presented in the next section. 
\end{remark}

\section{Colimits of small categories} \label{sec_colim}

We saw how, in studying colimits of sets, the relevant relations are best organized and presented in terms of certain categories. In studying colimits of small categories, we find that the relevant relations are best presented in terms of certain \textit{double categories}. We briefly review the latter and present an explicit construction of colimits in $\cat$. This challenge is mostly about book-keeping of indices and notation. 

\subsection{Two types of arrows? No problem.}

\begin{definition} [Double category]
    A double category is a category object in $\cat$.
\end{definition}

In order to observe double categories in nature, we ought to unpack the data encoded in the definition in terms of combinatorial elements and possibly depict the latter. Let $\bb{D}$ be a double category. The first piece of data encoded in $\bb{D}$ consists of two categories $\bb{D}_0$ and $\bb{D}_1$ related by source and target functors
$$s, t : \bb{D}_1 \rightarrow \bb{D}_0$$
We interpret this data as follows:
\begin{itemize}
    \item [$(\bullet)$] The \textit{objects} of $\bb{D}$ are the objects of the category $\bb{D}_0$.

    \item[$(\downarrow)$] The morphisms in $\bb{D}_0$ are the \textit{vertical} morphisms in $\bb{D}$. 

    \item[$(\xslashedrightarrow{})$] The objects of $\bb{D}_1$, being equipped with a source and target in $\bb{D}_0$, serve as another notion of morphism between objects of $\bb{D}$ which we call \textit{horizontal} morphisms. In order to distinguish the latter from vertical arrows, we depict them as $\xslashedrightarrow[]{}$.

    \item [$(\square)$] A morphism in $\bb{D}_1$, which we write as a double arrow $\rho : u \Rightarrow v$, is a square-shaped \textit{2-cell} 
    \[\begin{tikzcd}[column sep=scriptsize]
	a & b \\
	{a^\prime} & {b^\prime}
	\arrow["f"', from=1-1, to=2-1]
	\arrow["g", from=1-2, to=2-2]
	\arrow[""{name=0, anchor=center, inner sep=0}, "u", "\shortmid"{marking}, from=1-1, to=1-2]
	\arrow[""{name=1, anchor=center, inner sep=0}, "v"', "\shortmid"{marking}, from=2-1, to=2-2]
	\arrow["\rho", shorten <=6pt, shorten >=6pt, Rightarrow, from=0, to=1]
\end{tikzcd}\]
whose boundary is comprised of vertical and horizontal arrows. A cell as above records the fact that $s(\rho) = f$ and $t(\rho) = g$. 
\end{itemize}

We have two types of composition of arrows and 2-cells: vertical and horizontal. Vertical arrows compose as prescribed in $\bb{D}_0$. Moreover, 2-cells compose vertically as prescribed by $\bb{D}_1$:
\[\begin{tikzcd}[column sep=scriptsize]
	a & b \\
	{a^\prime} & {b^\prime} \\
	{a^{\prime\prime}} & {b^{\prime\prime}}
	\arrow["f"', from=1-1, to=2-1]
	\arrow["g", from=1-2, to=2-2]
	\arrow[""{name=0, anchor=center, inner sep=0}, "u", "\shortmid"{marking}, from=1-1, to=1-2]
	\arrow[""{name=1, anchor=center, inner sep=0}, "v"{description}, "\shortmid"{marking}, from=2-1, to=2-2]
	\arrow["{f^\prime}"', from=2-1, to=3-1]
	\arrow["{g^\prime}", from=2-2, to=3-2]
	\arrow[""{name=2, anchor=center, inner sep=0}, "w"', "\shortmid"{marking}, from=3-1, to=3-2]
	\arrow["\rho", shorten <=6pt, shorten >=6pt, Rightarrow, from=0, to=1]
	\arrow["{\rho^\prime}", shorten <=6pt, shorten >=6pt, Rightarrow, from=1, to=2]
\end{tikzcd} \xrightarrow{compose}
\begin{tikzcd}[column sep=scriptsize]
	a & b \\
	{a^\prime} & {b^\prime}
	\arrow[""{name=0, anchor=center, inner sep=0}, "u", "\shortmid"{marking}, from=1-1, to=1-2]
	\arrow["{g^\prime g}", from=1-2, to=2-2]
	\arrow["{f^\prime f}"', from=1-1, to=2-1]
	\arrow[""{name=1, anchor=center, inner sep=0}, "w"', "\shortmid"{marking}, from=2-1, to=2-2]
	\arrow["{\rho^\prime \rho}", shorten <=6pt, shorten >=6pt, Rightarrow, from=0, to=1]
\end{tikzcd}
\]
Horizontal composition in $\bb{D}$ is provided by the composition functor
$$\otimes  : \bb{D}_1 \times_{\bb{D}_0} \bb{D}_1 \rightarrow \bb{D}_1$$
which is part of the data for $\bb{D}$.
\[\begin{tikzcd}[column sep=scriptsize]
	a & b & c \\
	{a^\prime} & {b^\prime} & {c^\prime}
	\arrow[""{name=0, anchor=center, inner sep=0}, "u", "\shortmid"{marking}, from=1-1, to=1-2]
	\arrow["g"{description}, from=1-2, to=2-2]
	\arrow["f"', from=1-1, to=2-1]
	\arrow[""{name=1, anchor=center, inner sep=0}, "w"', "\shortmid"{marking}, from=2-1, to=2-2]
	\arrow["h", from=1-3, to=2-3]
	\arrow[""{name=2, anchor=center, inner sep=0}, "{u^\prime }", "\shortmid"{marking}, from=1-2, to=1-3]
	\arrow[""{name=3, anchor=center, inner sep=0}, "{w^\prime}"', "\shortmid"{marking}, from=2-2, to=2-3]
	\arrow["{ \rho}", shorten <=6pt, shorten >=6pt, Rightarrow, from=0, to=1]
	\arrow["\theta", shorten <=6pt, shorten >=6pt, Rightarrow, from=2, to=3]
\end{tikzcd} \xrightarrow{compose}
\begin{tikzcd}[column sep=scriptsize]
	a && c \\
	{a^\prime} && {c^\prime}
	\arrow["g"', from=1-1, to=2-1]
	\arrow["h", from=1-3, to=2-3]
	\arrow[""{name=0, anchor=center, inner sep=0}, "{u^\prime \otimes u}", "\shortmid"{marking}, from=1-1, to=1-3]
	\arrow[""{name=1, anchor=center, inner sep=0}, "{w^\prime \otimes w}"', "\shortmid"{marking}, from=2-1, to=2-3]
	\arrow["{\theta \otimes \rho}", shorten <=6pt, shorten >=6pt, Rightarrow, from=0, to=1]
\end{tikzcd}
\]
These compositions are \textit{associative} and \textit{unital} and they satisfy an \textit{interchange law} similar to that of 2-categories (which encodes the functoriallity of $\otimes$). 

We see from the above description, that double categories arise when we have two notions of morphism between a given class of objects. As the latter happens a lot, there are many interesting examples which we leave to the reader to explore. Recently, the literature on double categories has grown significantly. 

We also see that the definition of double category is symmetric with respect to the vertical and horizontal direction. However, as our notation suggests, our preference is to reserve the vertical direction for “traditional" morphisms and let the horizontal direction record “exotic" types of morphism (such as zig-zags for instance). 

\subsection{Double categories of zig-zags} \label{sec_dzig}

Let $\cc{J}$ be a category. We may form a double category $\bb{Z} \cc{J}$ whose objects and vertical morphisms are those of $\cc{J}$, while horizontal morphisms are zig-zags in $\cc{J}$. A 2-cell 
\[\begin{tikzcd}
	i & j \\
	{i^\prime} & {j^\prime}
	\arrow["{\rho_i}"', from=1-1, to=2-1]
	\arrow["{\rho_j}", from=1-2, to=2-2]
	\arrow[""{name=0, anchor=center, inner sep=0}, "{(\Bar{l}, \Bar{r})}", squiggly, from=1-1, to=1-2]
	\arrow[""{name=1, anchor=center, inner sep=0}, "{(\Bar{l}^\prime, \Bar{r}^\prime)}"', squiggly, from=2-1, to=2-2]
	\arrow["\rho", shorten <=6pt, shorten >=6pt, Rightarrow, from=0, to=1]
\end{tikzcd}\]
where $(\Bar{l}, \Bar{r}) : \cc{Z}^n \rightarrow \cc{J}$ and $(\Bar{l}^\prime, \Bar{r}^\prime) : \cc{Z}^m \rightarrow \cc{J}$,
consists of:
\begin{itemize}
    \item [-] A functor $\theta : \cc{Z}^n \rightarrow \cc{Z}^m$ which preserves the endpoints (the first and last target objects). 

    \item[-] A natural transformation $\rho : (\Bar{l}, \Bar{r}) \Rightarrow (\Bar{l}^\prime, \Bar{r}^\prime)\theta$
    \[\begin{tikzcd}
	{\mathcal{Z}^n} && {\mathcal{Z}^m} \\
	& {\mathcal{J}}
	\arrow[""{name=0, anchor=center, inner sep=0}, "{(\Bar{l}, \Bar{r})}"', from=1-1, to=2-2]
	\arrow["{(\Bar{l}^\prime, \Bar{r}^\prime)}", from=1-3, to=2-2]
	\arrow["\theta", from=1-1, to=1-3]
	\arrow["\rho"{description}, shorten <=8pt, shorten >=8pt, Rightarrow, from=0, to=1-3]
\end{tikzcd}\]
\end{itemize}
We can describe 2-cells in $\bb{Z} \cc{J}$ in more explicit terms. 

\begin{lemma}
    There is a bijection
    $$\cat_{*,*}(\cc{Z}^n , \cc{Z}^m) \cong \textbf{Surj}([n], [m])$$
    where the set on the right has as elements order preserving surjections $[n] \rightarrow [m]$ in $\D{}$.
\end{lemma}
 \begin{proof}
      This is easy to see, since we have
$$\cc{Z}^n \cong \underbrace{\cc{Z}^1 \vee \dots \vee \cc{Z}^1}_{n-\text{times}} $$
In creating a functor $\cc{Z}^n \rightarrow \cc{Z}^m$ which preserves endpoints, for the first copy of $\cc{Z}^1$ we have two choices: either map it to the first endpoint of $\cc{Z}^m$ or identically to the first copy of $\cc{Z}^1$ in $\cc{Z}^m$. And so on for the rest.
 \end{proof}

In light of the above lemma, the 2-cells in $\bb{Z}\cc{J}$ are generated by concatenation (horizontal composition) by the following two types of cells: 
\begin{itemize}
    \item [(i)] Cells indexed by the map $s^0 : [1] \rightarrow [0]$, which correspond to commutative squares
    \[\begin{tikzcd}
	& {j_1} \\
	i && j \\
	& k
	\arrow["l"', from=1-2, to=2-1]
	\arrow["r", from=1-2, to=2-3]
	\arrow["{\rho_i}"', from=2-1, to=3-2]
	\arrow["{\rho_j}", from=2-3, to=3-2]
\end{tikzcd}\]
Such a 2-cell has horizontal source $(l,r)$ and target the trivial zig-zag $\{ k\}$.

    \item[(ii)] Cells indexed by the identity $id : [1] \rightarrow [1]$, which are simply commutative diagrams of the form
    
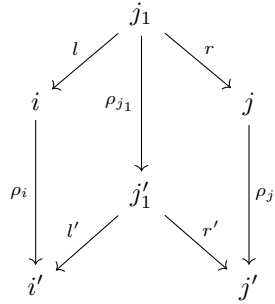
\[\begin{tikzcd}
	& {j_1} \\
	i && j \\
	& {j_1^\prime} \\
	{i^\prime} && {j^\prime}
	\arrow["l"', from=1-2, to=2-1]
	\arrow["r", from=1-2, to=2-3]
	\arrow["{l^\prime}"', from=3-2, to=4-1]
	\arrow["{r^\prime}", from=3-2, to=4-3]
	\arrow["{\rho_i}"', from=2-1, to=4-1]
	\arrow["{\rho_{j_1}}"', from=1-2, to=3-2]
	\arrow["{\rho_j}", from=2-3, to=4-3]
\end{tikzcd}\]
    Such a 2-cell has horizontal source $(l,r)$ and horizontal target $(l^\prime, r^\prime)$. 
\end{itemize}

Our second example of double category is a decorated version of $\bb{Z}\cc{J}$.

\begin{definition} [$F$-decoration]

Let $F : \cc{J} \rightarrow \cat$ be a diagram of small categories and $(\Bar{l}, \Bar{r}) : i \rightsquigarrow j$ be a zig-zag, say
    
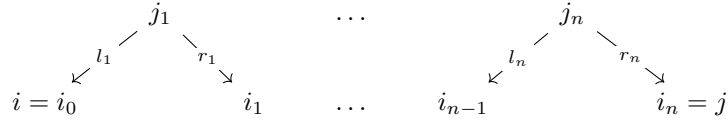
\[\begin{tikzcd}[column sep=scriptsize]
	& {j_1} && \dots && {j_n} \\
	{i = i_0} && {i_1} & \dots & {i_{n-1}} && {i_n = j}
	\arrow["{l_1}"{description}, from=1-2, to=2-1]
	\arrow["{r_1}"{description}, from=1-2, to=2-3]
	\arrow["{l_n}"{description}, from=1-6, to=2-5]
	\arrow["{r_n}"{description}, from=1-6, to=2-7]
\end{tikzcd}\]
An $F$-decoration of $(\Bar{l}, \Bar{r})$ consists of  sequence of objects $a_k \in \cc{C}_{j_k}$, $k = 1, \dots , n$, and a chain of morphisms 
    $$a \xrightarrow{f_0} \Tilde{l}_1(a_1) \ , \ \Tilde{r}_1(a_1) \xrightarrow{f_2} \Tilde{l}_2(a_2) \ , \dots , \ \Tilde{r}_n(a_n) \xrightarrow{f_n} b$$

    An $F$-decoration of the trivial zig-zag $\{i \}$ consists of a morphism $a \xrightarrow{f} b$ in $\cc{C}_i$.
\end{definition}

Let $DC = \bigoplus_{i \in \cc{J}} Ob(\cc{C}_i)$. $F$-decorated zig-zags provide a notion of morphism between elements of $DC$, as we  regard an $F$-decorated zig-zag as in the above definition to be an arrow 
\[\begin{tikzcd}
	{(i,a)} && {(j,b)}
	\arrow["{(\Bar{l}, \Bar{r}, \Bar{f})}", squiggly, from=1-1, to=1-3]
\end{tikzcd}\]

The composite $(i,a) \rightsquigarrow (j,b) \rightsquigarrow (k,c)$ is provided by concatenation of zig-zags in $\cc{J}$ and on decorations, by composing in $\cc{C}_j$ the last morphism in the first chain with the first morphisms in the second chain and then concatenating. The identity morphism on $(i,a)$ is provided by decorating the trivial zig-zag $\{i \}$ with the identity morphism $id_a$.

The double category $\bb{Z}\cc{J}$ acts on $F$-decorated zig-zags as follows. Given a 2-cell in $\bb{Z}\cc{J}$
\[\begin{tikzcd}
	i & j \\
	{i^\prime} & {j^\prime}
	\arrow["{\rho_i}"', from=1-1, to=2-1]
	\arrow["{\rho_j}", from=1-2, to=2-2]
	\arrow[""{name=0, anchor=center, inner sep=0}, "{(\Bar{l}, \Bar{r})}", squiggly, from=1-1, to=1-2]
	\arrow[""{name=1, anchor=center, inner sep=0}, "{(\Bar{l}^\prime, \Bar{r}^\prime)}"', squiggly, from=2-1, to=2-2]
	\arrow["\rho", shorten <=6pt, shorten >=6pt, Rightarrow, from=0, to=1]
\end{tikzcd}\]
and a decoration $(\Bar{l}, \Bar{r}, \Bar{f}) : (i,a) \rightsquigarrow (j,b)$, there is a decoration
\[\begin{tikzcd}
	{(i^\prime, a^\prime)} && {(j^\prime, b^\prime)}
	\arrow["{(\Bar{l}^\prime, \Bar{r}^\prime , \rho^* \Bar{f})}", squiggly, from=1-1, to=1-3]
\end{tikzcd}\]
where $a^\prime = \Tilde{\rho}_i(a)$ and $b^\prime = \Tilde{\rho}_j(b)$. The chain $\rho^* \Bar{f}$ is defined, roughly speaking, by pushing down the chain $\Bar{f}$ along the components of $\rho$ and then composing as prescribed by the map $\theta$ in $\D{}$ which indexes $\rho$. 

More precisely, the action of $\bb{Z} \cc{J}$ is defined on generating 2-cells as follows.
\begin{itemize}
    \item [(i)] In case the 2-cell $\rho$ is a commutative square,
    \[\begin{tikzcd}
	& {j_1} \\
	i && j \\
	& k
	\arrow["l"', from=1-2, to=2-1]
	\arrow["r", from=1-2, to=2-3]
	\arrow["{\rho_i}"', from=2-1, to=3-2]
	\arrow["{\rho_j}", from=2-3, to=3-2]
\end{tikzcd}\]
given a decoration $a \xrightarrow{f_0} \Tilde{l}(a_1) , \Tilde{r}(a_1) \xrightarrow{f_1} b$, we define $\rho^* \Bar{f}$ to be the decoration on the trivial zig-zag $\{ k\}$ provided by the composite
$$\rho^* \Bar{f} : a^\prime = \Tilde{\rho}_1(a) \xrightarrow{\Tilde{\rho}_if_0} \widetilde{\rho_il} (a_1) = \widetilde{\rho_jr}(a_1) \xrightarrow{\Tilde{\rho}_jf_1} \Tilde{\rho}_j(b) = b^\prime$$

\item[(ii)] In case $\rho$ is of the form 
\[\begin{tikzcd}
	& {j_1} \\
	i && j \\
	& {j_1^\prime} \\
	{i^\prime} && {j^\prime}
	\arrow["l"', from=1-2, to=2-1]
	\arrow["r", from=1-2, to=2-3]
	\arrow["{l^\prime}"', from=3-2, to=4-1]
	\arrow["{r^\prime}", from=3-2, to=4-3]
	\arrow["{\rho_i}"', from=2-1, to=4-1]
	\arrow["{\rho_{j_1}}"', from=1-2, to=3-2]
	\arrow["{\rho_j}", from=2-3, to=4-3]
\end{tikzcd}\]
we simply define $\rho^*(f_0, f_1) = (\Tilde{\rho}_i f_0, \Tilde{\rho}_j f_1)$
\end{itemize}
We extend the above action for general $\rho$ by concatenation in the horizontal direction. This is clearly functorial with respect to vertical composition of 2-cells in $\bb{Z} \cc{J}$, so that the term “action" is justified (see also the remark below). 

Lastly, we form a double category $\bb{Z}F$ with:
\begin{itemize}
    \item [$(\bullet)$] Set of objects $DC = \bigoplus_{i \in \cc{J}} Ob(\cc{C}_i)$.
    \item[$(\rightsquigarrow)$] Horizontal morphisms $F$-decorated zig-zags $(i,a) \rightsquigarrow (j,b)$.
    \item[$(\downarrow)$] Vertical morphisms $u^* : (i,a) \rightarrow (i^\prime , a^\prime)$ provided by morphisms $u : i \rightarrow i^\prime$ such that $\Tilde{u}(a) = a^\prime$.
    \item[$(\square)$] 2-cells
    \[\begin{tikzcd}
	{(i,a)} && {(j,b)} \\
	{(i^\prime, a^\prime)} && {(j^\prime, b^\prime)}
	\arrow[""{name=0, anchor=center, inner sep=0}, "{(\Bar{l}^\prime, \Bar{r}^\prime , \Bar{g})}"', squiggly, from=2-1, to=2-3]
	\arrow[""{name=1, anchor=center, inner sep=0}, "{(\Bar{l}, \Bar{r} , \Bar{f})}", squiggly, from=1-1, to=1-3]
	\arrow["{\rho_i^*}"', from=1-1, to=2-1]
	\arrow["{\rho_j^*}", from=1-3, to=2-3]
	\arrow["{\rho^*}", shorten <=4pt, shorten >=4pt, Rightarrow, from=1, to=0]
\end{tikzcd}\]
provided by 2-cells $\rho$ in $\bb{Z}\cc{J}$ such that $\rho^*\Bar{f} = \Bar{g}$.
\end{itemize}

\begin{note} [Functorial relations]
    At this point we can intuit how the double category $\bb{Z}F$ records the necessary relations for us to extract the colimit of $F$. For instance, its vertical category is precisely the category of elements $\cc{E}(ObF)$ of the restriction of $F$ on objects, and hence records the correct identifications on objects. The 2-cells record the identifications on morphisms  in a manner which is \textit{functorial} with respect to the identifications on objects.
\end{note}

\begin{remark} [Double fibrations]
    There is a forgetful double functor $\bb{Z} F \rightarrow \bb{Z} \cc{J}$. This functor is an example of a \textit{discrete double fibration}. The latter are simply category objects in the category of discrete fibrations and have been studied in \cite{lambert2021discrete}. In this vein, a 2-cell $\rho^*$ in $\bb{Z}F$ is the cocartesian lift of the corresponding 2-cell $\rho$ in $\bb{Z} \cc{J}$. Such fibrations correspond to double functors into a suitable double category of sets (see \cite{lambert2021discrete}). 
\end{remark}

We are ultimately interested in encoding relations via double categories. Just as in the case of relations encoded by categories, they suffer from a lack of symmetry. For a category $\cc{E}$, this is resolved by forming the category of zig-zags $\cc{Z}\cc{E}$. A similar construction is possible for double categories.

Let $\bb{D}$ be a double category. There is a double category $\cc{Z} \bb{D}$ with
$$(\cc{Z} \bb{D})_0 = \cc{Z} (\bb{D}_0) \ \  , \ \  (\cc{Z} \bb{D})_1 = \cc{Z} (\bb{D}_1)$$
In other words, 2-cells in $\cc{Z} \bb{D}$ are zig-zags (in the vertical direction) of 2-cells in $\bb{D}$. For instance, 
\[\begin{tikzcd}
	a & b \\
	{a^\prime} & {b^\prime} \\
	{a^{\prime\prime}} & {b^{\prime\prime}}
	\arrow[from=2-1, to=1-1]
	\arrow[from=2-2, to=1-2]
	\arrow[""{name=0, anchor=center, inner sep=0}, "\shortmid"{marking}, from=1-1, to=1-2]
	\arrow[""{name=1, anchor=center, inner sep=0}, "\shortmid"{marking}, from=2-1, to=2-2]
	\arrow[""{name=2, anchor=center, inner sep=0}, "\shortmid"{marking}, from=3-1, to=3-2]
	\arrow[from=2-1, to=3-1]
	\arrow[from=2-2, to=3-2]
	\arrow[shorten <=6pt, shorten >=6pt, Rightarrow, from=1, to=0]
	\arrow[shorten <=6pt, shorten >=6pt, Rightarrow, from=1, to=2]
\end{tikzcd}\]
In this double category, vertical composition in provided by concatenation, while horizontal composition is provided by composing in $\bb{D}$.

We are interested in the double categories
$$\Tilde{\bb{Z}} \cc{J} = \cc{Z}(\bb{Z} \cc{J}) \ \ , \ \ \Tilde{\bb{Z}} F = \cc{Z}(\bb{Z} F)$$
The 2-cells in these categories are depicted as
\[\begin{tikzcd}
	i & j \\
	{i^\prime} & {j^\prime}
	\arrow[""{name=0, anchor=center, inner sep=0}, squiggly, from=1-1, to=1-2]
	\arrow[""{name=1, anchor=center, inner sep=0}, squiggly, from=2-1, to=2-2]
	\arrow[squiggly, from=1-1, to=2-1]
	\arrow[squiggly, from=1-2, to=2-2]
	\arrow[shorten <=6pt, shorten >=6pt, Rightarrow, squiggly, from=0, to=1]
\end{tikzcd} \ \ , \ \ 
\begin{tikzcd}
	{(i,a)} & {(j,b)} \\
	{(i^\prime, a^\prime)} & {(j^\prime, b^\prime)}
	\arrow[""{name=0, anchor=center, inner sep=0}, squiggly, from=1-1, to=1-2]
	\arrow[""{name=1, anchor=center, inner sep=0}, squiggly, from=2-1, to=2-2]
	\arrow[squiggly, from=1-1, to=2-1]
	\arrow[squiggly, from=1-2, to=2-2]
	\arrow[shorten <=6pt, shorten >=6pt, Rightarrow, squiggly, from=0, to=1]
\end{tikzcd}
\]

In $\Tilde{\bb{Z}} \cc{J}$, both horizontal and vertical morphisms are zig-zags in $\cc{J}$. In $\Tilde{\bb{Z}} F$, the horizontal morphisms are $F$-decorated zig-zags, while the vertical morphisms are zig-zags in $\cc{E}(ObF)$. We regard the latter as $F$-decorated by identities. These double categories are well-behaved and have pleasant properties which we discuss in Appendix \ref{app_nice}.

\begin{remark}
    We still have a forgetful functor $\Tilde{\bb{Z}} F \rightarrow \Tilde{\bb{Z}} \cc{J}$, but this functor fails to be a fibration, so we cannot speak of an action.
\end{remark}

\subsection{A model for colimits} \label{sec_colimthm}

For a diagram of categories $$F : \cc{J} \rightarrow \cat$$
we construct an explicit category $\cc{C}$ and demonstrate that it serves as a colimit for $F$. 

\begin{itemize}
    \item[$(\bullet)$] 
    We define 
    $$Ob(\cc{C}) = \colim_{i \in \cc{J}} Ob(\cc{C}_i)$$
    Hence, objects are represented by pairs $(i,a) \in DC$ subject to the relation provided by morphisms in the category of elements $\cc{E}(ObF)$.

     \item [$(\rightarrow)$] 
    Let $[a], [b] \in Ob(\cc{C})$. A morphism $[a] \rightarrow [b]$ in $\cc{C}$ consists of an $F$-decorated zig-zag 
    \[\begin{tikzcd}
	{(i,a)} && {(j,b)}
	\arrow["{(\Bar{l}, \Bar{r} , \Bar{f})}", squiggly, from=1-1, to=1-3]
\end{tikzcd}\]
    where $(i,a)$ and $(j,b)$ are representatives of $[a]$ and $[b]$, subject to the relation determined by 2-cells in the double category $\bb{Z}F$. In other words, in case there is a 2-cell
    \[\begin{tikzcd}
	{(i,a)} && {(j,b)} \\
	{(i^\prime, a^\prime)} && {(j^\prime, b^\prime)}
	\arrow[""{name=0, anchor=center, inner sep=0}, "{(\Bar{l}^\prime, \Bar{r}^\prime , \Bar{g})}"', squiggly, from=2-1, to=2-3]
	\arrow["{\rho_j^*}", from=1-3, to=2-3]
	\arrow[""{name=1, anchor=center, inner sep=0}, "{(\Bar{l}, \Bar{r} , \Bar{f})}", squiggly, from=1-1, to=1-3]
	\arrow["{\rho_i^*}"', from=1-1, to=2-1]
	\arrow["{\rho^*}", shorten <=4pt, shorten >=4pt, Rightarrow, from=1, to=0]
\end{tikzcd}\]
the top and bottom zig-zags represent the same morphism $[a] \rightarrow [b]$.

\item [$(\circ)$] Let $[a] \rightarrow [b]$ and $[b] \rightarrow [c]$ be composable morphisms in $\cc{C}$. Assume these morphisms are represented by $F$-decorated zig-zags $(i,a) \rightsquigarrow (j,b)$ and $(j^\prime, b^\prime) \rightsquigarrow (k, c)$. Choose an $ObF$-decorated zig-zag $(j,b) \rightsquigarrow (j^\prime , b^\prime)$ (we know there exists at least one such zig-zag), which we regard as $F$-decorated by identity morphisms.
   
    We define the composite $[a] \rightarrow [c]$ to be represented by the composite of $F$-decorated zig-zags
    \[\begin{tikzcd}
	{(i,a)} & {(j,b)} \\
	& {(j^\prime, b^\prime)} & {(c,k)}
	\arrow[squiggly, from=1-1, to=1-2]
	\arrow[squiggly, from=1-2, to=2-2]
	\arrow[squiggly, from=2-2, to=2-3]
\end{tikzcd}\]

\item[$(=)$] For $[a] \in \cc{C}$, we choose a representative $(i,a)$ and define the identity morphism $id_{[a]}$ to be represented by $(\{ i\}, id_a)$. 

\end{itemize}

\begin{proposition} \label{prop_good}
    As defined above, the category $\cc{C}$ is well-defined.
\end{proposition}

The proof of this proposition relies on moving zig-zags around. We present it systematically in Appendix \ref{app_prop}.

\begin{theorem} \label{thm_colim}
    There is an isomorphism
    $$\cc{C} \cong \colim_{i \in \cc{J}} \cc{C}_i$$
\end{theorem}

\begin{proof}
    We demonstrate that $\cc{C}$ has the structure of a cocone over $F$ by constructing a natural transformation $\phi : F \Rightarrow \cc{C}$ and then verify that this cocone is universal. 

    For $i \in \cc{J}$, let the functor $$\phi_i : \cc{C}_i \rightarrow \cc{C}$$ be defined by $\phi_i(a) = [a]$ on objects. For a morphism $f: a \rightarrow b$, we let $\phi_i(f)$ be the morphism in $\cc{C}$ represented by the decoration $(\{ i\}, f) : (i,a) \rightsquigarrow (i,b) $ on the trivial zig-zag $\{i \}$ by the morphism $f$.

    For a morphism $u : i \rightarrow j$ in $\cc{J}$, the triangle 
    \[\begin{tikzcd}
	{\mathcal{C}_i} && {\mathcal{C}_j} \\
	& {\mathcal{C}}
	\arrow["{\phi_i}"', from=1-1, to=2-2]
	\arrow["{\phi_j}", from=1-3, to=2-2]
	\arrow["{\tilde{u}}", from=1-1, to=1-3]
\end{tikzcd}\]
    commutes. This is clear on objects. For a morphism $f$ in $\cc{C}_i$, the morphism $u$, regarded an identity 2-cell in $\bb{Z}\cc{J}$, induces a 2-cell $u^* : (\{ i\}, f) \rightarrow (\{ j\}, \Tilde{u}(f))$ in $\bb{Z}F$ which identifies its source and target in $\cc{C}$.

    To verify the universal property, let $\cc{D}$ be a category with the structure of a cocone $\psi : F \Rightarrow \cc{D}$. We construct a functor $H : \cc{C} \rightarrow \cc{D}$ which makes the following triangle commute
    \[\begin{tikzcd}
	F && {\mathcal{C}} \\
	& {\mathcal{D}}
	\arrow["\psi"', Rightarrow, from=1-1, to=2-2]
	\arrow["H", dotted, from=1-3, to=2-2]
	\arrow["\phi", Rightarrow, from=1-1, to=1-3]
\end{tikzcd}\]
On objects, given $[a] \in \cc{C}$, we choose a representative $(i,a)$ and define 
$$H([a]) = \psi_i(a)$$
This is well-defined, since any two representatives are connected by an $ObF$-decorated zig-zag which maps under $\psi$ to a chain of identity morphisms in $\cc{D}$.

Let $[a] \rightarrow [b]$ be a morphism in $\cc{C}$. For a representative $(\Bar{l}, \Bar{r}, \Bar{f}) : (i,a) \rightsquigarrow (j,b)$, we obtain a chain of composable morphisms $\psi(\Bar{f})$ in $\cc{D}$ (in virtue of $\psi$ being natural). We define
$$H(\Bar{l}, \Bar{r}, \Bar{f}) = \circ (\psi(\Bar{f}))$$
where the expression on the right denoted the composition of the chain. 

This mapping is well-defined. Indeed, assume we have a 2-cell in $\bb{Z}F$
\[\begin{tikzcd}
	{(i,a)} && {(j,b)} \\
	{(i^\prime, a^\prime)} && {(j^\prime, b^\prime)}
	\arrow[""{name=0, anchor=center, inner sep=0}, "{(\Bar{l}^\prime, \Bar{r}^\prime , \Bar{g})}"', squiggly, from=2-1, to=2-3]
	\arrow[""{name=1, anchor=center, inner sep=0}, "{(\Bar{l}, \Bar{r} , \Bar{f})}", squiggly, from=1-1, to=1-3]
	\arrow["{\rho_i^*}"', from=1-1, to=2-1]
	\arrow["{\rho_j^*}", from=1-3, to=2-3]
	\arrow["{\rho^*}", shorten <=4pt, shorten >=4pt, Rightarrow, from=1, to=0]
\end{tikzcd}\]
In virtue of this cell being indexed by a cell $\rho : (\Bar{l}, \Bar{r}) \Rightarrow (\Bar{l}^\prime, \Bar{r}^\prime)$ in $\bb{Z} \cc{J}$, which is a natural transformation of zig-zags in $\cc{J}$, and the fact that $\Bar{g} = \rho^*\Bar{f}$, we must have $\circ(\psi(\Bar{f})) = \circ(\psi(\Bar{g}))$.

It is clear that $H$ preserves composition and that $H\phi = \psi$. Hence the universality of $\cc{C}$.
\end{proof}

\begin{note} [The double category of elements]
    In the previous section we mentioned how the colimit of a diagram of a sets $F : \cc{J} \rightarrow \set$ is isomorphic to the set of connected components $\pi_0 \cc{E}(F)$ of the category of elements of $F$. But, we did not follow up in categorifying this idea, even though we did so for the categories of zig-zags introduced in the same section. 

    One way to achieve such a categorification is by introducing the \textit{double category of elements} $\bb{E}(F)$ associated to a diagram of categories $F : \cc{J} \rightarrow \cat$. This category is defined with:
    \begin{itemize}
    \item [$(\bullet)$] Set of objects $DC = \bigoplus_{i \in \cc{J}} Ob(\cc{C}_i)$.

    \item[$(\downarrow)$] Vertical category the direct sum $\bigoplus_{i \in \cc{J}} \cc{C}_i$.

    \item[$(\mapsto)$] A horizontal morphism $u^* : (i, a) \mapsto (j,a^\prime )$ being a morphism in the category of elements $\cc{E}(ObF)$, i.e. a morphism $u : i \rightarrow j$ in $\cc{J}$ such that $a \mapsto a^\prime$ under $\Tilde{u}$. 

    \item[$(\square)$]  2-cells 
    \[\begin{tikzcd}
	{(i,a)} & {(j, a^\prime)} \\
	{(i,b)} & {(j, b^\prime)}
	\arrow["f"', from=1-1, to=2-1]
	\arrow[""{name=0, anchor=center, inner sep=0}, maps to, from=1-1, to=1-2]
	\arrow[""{name=1, anchor=center, inner sep=0}, maps to, from=2-1, to=2-2]
	\arrow["g", from=1-2, to=2-2]
	\arrow["{u^*}"{description}, draw=none, from=0, to=1]
\end{tikzcd}\]
in case of a morphism $u : i \rightarrow j$ in $\cc{J}$ such that $\Tilde{u}(f) = g$.
\end{itemize}

If we consider the category of elements as an attempt to depict a diagram of sets, the double category of elements displayed above is an attempt to depict a diagram of categories. 

In general, for a double category $\bb{D}$ we may define the \textit{vertical category of connected components} $\pi_0^v \bb{D}$. The functor 
$$\pi_0^v : \dcat \rightarrow \cat$$
from double categories to categories can be defined to be the left adjoint to the functor $\bb{V} : \cat \rightarrow \dcat$ which regards a category $\cc{A}$ as a double category $\bb{V} \cc{A}$ with $\cc{A}$ in the vertical direction and the horizontal direction trivial. 

For general double categories $\bb{D}$, it is certainly tedious to describe $\pi_0^v \bb{D}$ explicitly, but we can understand from the universal  property that any 2-cell
\[\begin{tikzcd}
	a & b \\
	c & d
	\arrow["f"', from=1-1, to=2-1]
	\arrow[""{name=0, anchor=center, inner sep=0}, "\shortmid"{marking}, from=1-1, to=1-2]
	\arrow[""{name=1, anchor=center, inner sep=0}, "\shortmid"{marking}, from=2-1, to=2-2]
	\arrow["g", from=1-2, to=2-2]
	\arrow[shorten <=4pt, shorten >=4pt, Rightarrow, from=0, to=1]
\end{tikzcd}\]
provides an identification $a \sim b$, $c \sim d$ and $f \sim g$.
For this reason, we understand that 
$$\colim F \cong \pi_0^v \bb{E}(F)$$

In some sense, the double category $\bb{Z}F$ is a way of symmetrizing $\bb{E}(F)$ in order to have an explicit colimit.

\end{note}

\section{The necklace theorem}

Let $X$ be a simplicial set. 
We adopt the following conventions:
\begin{itemize}
    \item[-] For an $n$-simplex $x \in X_n$, let $\D{n}(x)$ denote the copy of $\D{n}$ indexed by the simplex $x$.

    \item [-] $V(x)$ denotes the set of vertices of the simplex $\D{n}(x)$. We regard it as an ordered set.
    
    \item[-] For vertices $a, b \in V(x)$ we denote by $x_{\ra{ab}}$ the subsimplex of $x$ spanned by vertices “between" $a$ and $b$. 
\end{itemize}

Let $\cc{S}(X)$ denote the category of simplices of $X$ (as defined in the introduction). There is a forgetful functor
\[
\begin{matrix}
     \spx & \rightarrow & \sset \\
     x & \mapsto & \D{n}(x)
\end{matrix}
\]
and by general presheaf theory we have
$$X \cong \colim_{x \in X_n} \D{n}(x)$$

In a similar fashion we denote  $\bb{\D{n}}(x)$ the copy of the simplicial category $\bb{\D{n}}$ indexed by the simplex $x \in X_n$. As mentioned in the introduction, for a fixed $[p] \in \D{}$, the category of $\bb{\D{n}}(x)_p$ is defined with:
\begin{itemize}
    \item [$(\bullet)$] Set of objects $V(x)$.
    \item[$(\triangle)$] Given $a , b \in V(x)$, a $p$-arrow $\ra{U} : a \rightarrow b$ is a flag $ U^0 \subseteq \dots \subseteq U^p$ of subsets of $V(x_{\ra{ab}})$ such that $a,b \in U^0$.
    \item[$(\circ)$] Composition is provided by union of sets. 
\end{itemize}

Fix some $[p] \in \D{}$. The category of $p$-arrows $\ff{C}(X)_p$ is the colimit of the diagram of categories
\[
\begin{matrix}
    \chi_p : & \spx & \rightarrow & \cat \\
    & x & \mapsto & \bb{\D{n}}(x)_p
\end{matrix}
\]
Therefore, the general theory developed in the previous section to describe $\cat$-valued colimits applies. We interpret in the language of simplices.

First of all, on objects, we have 
$$Ob(\ff{C}(X)_p) = X_0$$
In light of our notation, an object $a \in X_0$ is represented by a pair $(x, a)$ where $x \in X_n$ and $a \in V(x)$.  

The bookkeeping of indices for the above colimit is carried out by the double category $\bb{Z}\spx$, which we interpret as follows. Consider a zig-zag $T : [n] \rightsquigarrow [m]$ in $\D{}$, say
\[\begin{tikzcd}
	& {[m_1]} && \dots && {[m_k]} \\
	{[n]} && {[n_1]} & \dots & {[n_{k-1}]} && {[m]}
	\arrow["{l_1}"', from=1-2, to=2-1]
	\arrow["{r_1}", from=1-2, to=2-3]
	\arrow["{l_k}"', from=1-6, to=2-5]
	\arrow["{r_k}", from=1-6, to=2-7]
\end{tikzcd}\]
Being a diagram in $\D{}$, this zig-zag can be realized as a simplicial set
$$|T| = \D{n} +_{\D{m_1}} \dots +_{\D{m_{k-1}}} \D{m}$$
Moreover, any 2-cell in $\bb{Z}\D{}$ 
\[\begin{tikzcd}
	{[n]} & {[m]} \\
	{[n^\prime]} & {[m^\prime]}
	\arrow[""{name=0, anchor=center, inner sep=0}, "T", squiggly, from=1-1, to=1-2]
	\arrow[""{name=1, anchor=center, inner sep=0}, "S"', squiggly, from=2-1, to=2-2]
	\arrow[from=1-1, to=2-1]
	\arrow[from=1-2, to=2-2]
	\arrow["\rho", shorten <=4pt, shorten >=4pt, Rightarrow, from=0, to=1]
\end{tikzcd}\]
is represented by a morphism of simplicial sets $\rho : |T| \rightarrow |S|$ (beware though, there are more simplicial maps $|T| \rightarrow |S|$ than 2-cells as above). 

In this vein, a zig-zag in $\spx$ is a map of simplicial sets $|T| \rightarrow X$ for some zig-zag $T$ in $\D{}$, and a 2-cell in $\bb{Z} \spx$ 
\[\begin{tikzcd}
	x && y \\
	{x^\prime} && {x^\prime}
	\arrow[""{name=0, anchor=center, inner sep=0}, "{|T| \rightarrow X}", squiggly, from=1-1, to=1-3]
	\arrow[""{name=1, anchor=center, inner sep=0}, "{|S| \rightarrow X}"', squiggly, from=2-1, to=2-3]
	\arrow[from=1-1, to=2-1]
	\arrow[from=1-3, to=2-3]
	\arrow["\rho", shorten <=4pt, shorten >=4pt, Rightarrow, from=0, to=1]
\end{tikzcd}\]
is represented by a simplicial map over $X$
\[\begin{tikzcd}
	{|T|} && {|S|} \\
	& X
	\arrow["\rho", from=1-1, to=1-3]
	\arrow[from=1-3, to=2-2]
	\arrow[from=1-1, to=2-2]
\end{tikzcd}\]
Again, there are more such commutative triangles than 2-cells as above.

Let $(|T| \rightarrow X) : x \rightsquigarrow y$ be a zig-zag in $\cc{Z}\spx$, say
\[\begin{tikzcd}
	& {y_0} && \dots && {y_k} \\
	{x = x_0} && {x_1} & \dots & {x_{k-1}} && {x_k = y}
	\arrow["{l_1}"', from=1-2, to=2-1]
	\arrow["{r_1}", from=1-2, to=2-3]
	\arrow["{l_k}"', from=1-6, to=2-5]
	\arrow["{r_k}", from=1-6, to=2-7]
\end{tikzcd}\]
A $\chi_p$-decoration of such a zig-zag consists of a choice of vertices $a \in V(x)$, $a_i \in V(y_i)$ and $b \in V(y)$ together with a chain  $\ra{U}_0 : a \rightarrow a_1, \dots , \ra{U}_k : a_k \rightarrow b$, where $\ra{U}_i$ is a morphism in $\bb{\D{n_i}}(x_i)_p$.

Observe that, in constructing a decoration, we do not make use of the simplices $y_i$ except in choosing the vertices $a_i$. Moreover, in constructing a morphism $\ra{U}_i : a_i \rightarrow a_{i+1}$, we only make use of the subsimplex of $x_i$ spanned by the vertices between $a_i$ and $a_{i+1}$. In some sense, we see the notion of \textit{necklace} originating from this observation. 

\begin{definition}
\begin{itemize}
    \item[-] (Necklace) A necklace $N$ is a zig-zag in $\D{}$ of the form
    \[\begin{tikzcd}
	& {[0]} && \dots && {[0]} \\
	{[n]} && {[n_1]} & \dots & {[n_{k-1}]} && {[m]}
	\arrow["{\omega_1}"', from=1-2, to=2-1]
	\arrow["{\alpha_1}", from=1-2, to=2-3]
	\arrow["{\omega_k}"', from=1-6, to=2-5]
	\arrow["{\alpha_k}", from=1-6, to=2-7]
\end{tikzcd}\]
    where the maps $\alpha$ and $\omega$ indicate the initial and terminal vertex inclusions.

    \item [-] (Necklace in $X$).  A necklace in $X$ is a zig-zag in $\spx$ of the form $|N| \rightarrow X$ for some necklace $N$ in $\cc{Z}\D{}$, say
    \[\begin{tikzcd}
	& {a_1} && \dots && {a_k} \\
	x && {x_1} & \dots & {x_{k-1}} && y
	\arrow["{\omega_1}"', from=1-2, to=2-1]
	\arrow["{\alpha_1}", from=1-2, to=2-3]
	\arrow["{\omega_k}"', from=1-6, to=2-5]
	\arrow["{\alpha_k}", from=1-6, to=2-7]
\end{tikzcd}\]

\item[-] (Proper decoration). A $\chi_p$-decorated zig-zag $(x,a) \rightsquigarrow (y,b)$ is called proper if $a$ is the initial vertex of $x$ and $b$ the terminal vertex of $y$.  
\end{itemize}   
\end{definition}

The realization of a necklace $N : [n] \rightsquigarrow [m]$ in $\D{}$ may be written as
$$|N| = \D{n} \vee \dots \vee \D{m}$$
Hence, we recover the definition of necklace given by Dugger and Spivak in \cite{dugger2011rigidification}. 
We may define beads, joins and vertices as in the introduction. 
 The novelty in our treatment is only in interpreting necklaces as zig-zags, which allows us to use the double categorical framework developed earlier.

Let $\bb{N}_p X$ denote the full sub-double category of $\bb{Z}\chi_p$ in which the horizontal morphism are properly decorated necklaces. The following proposition is a combinatorial unpacking of the structure of $\bb{N}_p X$ in the language of simplicial sets.

\begin{proposition} \label{prop_translate}
    \begin{itemize}
        \item [(i)] Let $N$ and $M$ be two necklaces in $\D{}$, say $|N| = \D{n_0} \vee \dots \vee \D{n_k}$ and $|M| = \D{m_0} \vee \dots \vee\D{m_l}$.
        
        There is a bijection between the set of 2-cells  $(\bb{Z} \D{})_1 (N, M)$ and 
        the set of simplicial maps $ |N| \rightarrow |M| $ which restrict to maps $\D{n_0} \rightarrow \D{m_0}$ and $\D{n_k} \rightarrow \D{m_l}$. 

        \item[(ii)] Let $|N| \rightarrow X$ be a necklace in $X$. There is a bijection between  proper $\chi_p$-decorations of this necklace (when regarded as a zig-zag in $\spx$) and flags $\ra{U} = ( U^0 \subseteq \dots \subseteq U^p )$ of subsets of $V_N$ such that $J_N \subseteq U^0$.

        \item[(iii)] There is a bijection between 2-cells in $\bb{N}_p X$ 
        \[\begin{tikzcd}
	{(a,x)} && {(a,x)} \\
	{(a,x^\prime)} && {(a,x^\prime)}
	\arrow[""{name=0, anchor=center, inner sep=0}, "{(|N| \rightarrow X, \ra{U})}", squiggly, from=1-1, to=1-3]
	\arrow[""{name=1, anchor=center, inner sep=0}, "{(|M| \rightarrow X, \ra{V})}"', squiggly, from=2-1, to=2-3]
	\arrow[from=1-1, to=2-1]
	\arrow[from=1-3, to=2-3]
	\arrow["{\rho^*}", shorten <=4pt, shorten >=4pt, Rightarrow, from=0, to=1]
\end{tikzcd}\]
and commutative triangles in $\sset_{*,*}$
\[\begin{tikzcd}
	{|N|} && {|M|} \\
	& {X_{a,b}}
	\arrow["\rho", from=1-1, to=1-3]
	\arrow[from=1-1, to=2-2]
	\arrow[from=1-3, to=2-2]
\end{tikzcd}\]
such that $\rho(\ra{U}) = \ra{V}$.

    \end{itemize}
\end{proposition}

\begin{proof}
    \begin{itemize}
        \item [(i)] Let $$\phi : |N| \rightarrow |M|$$ be a morphism which restricts to maps $\phi_0 : \D{n_0} \rightarrow \D{m_0}$ and $\phi_k : \D{n_k} \rightarrow \D{m_l}$.

        In general, for $i = 0 ,\dots , k$, the map $\phi$ restricts to a simplicial map $\phi_i : \D{n_i} \rightarrow \D{m_j}$ for some $j$ which depends on $i$. This gives rise to a morphism $\theta : [k] \rightarrow [l]$ in $\D{}$.

        The map $\theta$ has to be surjective since $\theta(0) = 0$ and $\theta(k) = l$. Therefore, the maps $\phi_i$ serve as components for a natural transformation
        \[\begin{tikzcd}[column sep=scriptsize]
	{\mathcal{Z}^k} && {\mathcal{Z}^l} \\
	& \Delta
	\arrow[""{name=0, anchor=center, inner sep=0}, "N"', from=1-1, to=2-2]
	\arrow["M", from=1-3, to=2-2]
	\arrow["\theta", from=1-1, to=1-3]
	\arrow["\phi"{description}, shorten <=8pt, Rightarrow, from=0, to=1-3]
\end{tikzcd}\]
    which, by definition, is a 2-cell in $\bb{Z}\D{}$
    \[\begin{tikzcd}[column sep=scriptsize]
	{[n_0]} & {[n_k]} \\
	{[m_0]} & {[m_l]}
	\arrow[""{name=0, anchor=center, inner sep=0}, "N", squiggly, from=1-1, to=1-2]
	\arrow[""{name=1, anchor=center, inner sep=0}, "M"', squiggly, from=2-1, to=2-2]
	\arrow[from=1-1, to=2-1]
	\arrow[from=1-2, to=2-2]
	\arrow["\phi", shorten <=4pt, shorten >=4pt, Rightarrow, from=0, to=1]
\end{tikzcd}\]
The above steps are reversible, hence the desired bijection.

\item[(ii)] This is clear, since any chain $\ra{U}$ as above may be split along the joins to produce a proper decoration, and vice versa.

\item[(iii)] 2-cells $\rho^*$ as in the statement are indexed by 2-cells $\rho$ in $\bb{Z} \D{}$
\[\begin{tikzcd}[column sep=scriptsize]
	x & y \\
	{x^\prime} & {y^\prime}
	\arrow[""{name=0, anchor=center, inner sep=0}, "N", squiggly, from=1-1, to=1-2]
	\arrow[""{name=1, anchor=center, inner sep=0}, "M"', squiggly, from=2-1, to=2-2]
	\arrow[from=1-1, to=2-1]
	\arrow[from=1-2, to=2-2]
	\arrow["\rho", shorten <=4pt, shorten >=4pt, Rightarrow, from=0, to=1]
\end{tikzcd}\]
By (i), such 2-cells are precisely simplicial maps $\rho : |N| \rightarrow |M|$ which restrict in the first and last component. 

However, the fact that $\rho$ indexes $\rho^*$ forces the restrictions on $x$ and $y$ to preserve the first and last vertex respectively. Hence, such a morphism $\rho$ is just a morphism of bipointed simplicial sets over $X_{a,b}$. The fact that $\rho(\ra{U}_*) = \ra{V}_*$ follows by (ii). 

Conversely, any bipointed simplicial map $\rho : |N| \rightarrow |M|$ will automatically restrict to morphisms on $x$ and $y$. Hence the bijection.
    \end{itemize}
\end{proof}

\begin{construction}[Necklace replacement functor] \label{const_ncsrep}

We construct a functor 
$$N_* : (\bb{Z} \chi_p)_1 \rightarrow (\bb{N}_p X)_1$$

Let $(|T| \rightarrow X, \ra{U}_*) : (x,a) \rightsquigarrow (y,b)$ be an object of $\bb{Z}(\chi_p)_1$ (which is a horizontal morphism in $\bb{Z} \chi_p$). Say, the indexing zig-zag in $\cc{Z} \cc{S}(X)$ is of the form
\[\begin{tikzcd}
	& {y_1} && {y_2} & \dots & {y_k} \\
	{x = x_0} && {x_1} && \dots && {x_k = y}
	\arrow["{l_1}"', from=1-2, to=2-1]
	\arrow["{r_1}", from=1-2, to=2-3]
	\arrow["{l_2}"', from=1-4, to=2-3]
	\arrow["{r_k}", from=1-6, to=2-7]
\end{tikzcd}\]
and the decoration is given by vertices $a \in V(x)$, $a_i \in V(y_i)$ and $b \in V(y)$ together with a chain of morphisms $\ra{U}_0 : a \rightarrow a_1, \dots , \ra{U}_k : a_{k-1} \rightarrow b$ ($\ra{U}_i$ is a morphism in the category $\bb{\D{n_i}}(x_i)_p$, the dimension of $x_i$ being $n_i$). 

There is a necklace $|N_T| \rightarrow X$, equipped with an inclusion  2-cell $\epsilon_T$ in $\bb{Z} \spx$ which is displayed by the following commutative diagram
\[\begin{tikzcd}[column sep=scriptsize]
	& {a_1} && {a_2} & \dots & {a_k} \\
	{x_{\ra{aa_1}}} && {x_{1, \ra{a_1a_2}}} && \dots && {y_{\ra{a_k b}}} \\
	& {y_0} && {y_1} & \dots & {y_k} \\
	x && {x_1} && \dots && y
	\arrow["{l_1}"', from=3-2, to=4-1]
	\arrow["{r_1}", from=3-2, to=4-3]
	\arrow["{l_2}"', from=3-4, to=4-3]
	\arrow["{r_k}", from=3-6, to=4-7]
	\arrow["\omega"', from=1-2, to=2-1]
	\arrow["\alpha", from=1-2, to=2-3]
	\arrow["\omega"', from=1-4, to=2-3]
	\arrow[hook', from=2-1, to=4-1]
	\arrow[hook', from=1-2, to=3-2]
	\arrow[hook', from=2-3, to=4-3]
	\arrow[hook', from=1-4, to=3-4]
	\arrow[hook', from=1-6, to=3-6]
	\arrow[hook', from=2-7, to=4-7]
	\arrow["\alpha", from=1-6, to=2-7]
\end{tikzcd}\]
It is clear that the decoration $\ra{U}_*$ originates from the the necklace on top, so that we have a 2-cell in $\bb{Z} \chi_p$ 
\[\begin{tikzcd}
	{(x_{\ra{a a_1}}, a)} && {(y_{\ra{a_k b}}, b)} \\
	{(x,a)} && {(y,b)}
	\arrow[""{name=0, anchor=center, inner sep=0}, "{(|N_T| \rightarrow X, \ra{U}_*)}", squiggly, from=1-1, to=1-3]
	\arrow[""{name=1, anchor=center, inner sep=0}, "{(|T| \rightarrow X, \ra{U}_*)}"', squiggly, from=2-1, to=2-3]
	\arrow[from=1-1, to=2-1]
	\arrow[from=1-3, to=2-3]
	\arrow["\epsilon_T", shorten <=4pt, shorten >=4pt, Rightarrow, from=0, to=1]
\end{tikzcd}\]

The construction $(|T| \rightarrow X, \ra{U}_*) \mapsto (|N_T| \rightarrow X, \ra{U}_*)$ is clearly functorial.
Moreover, the 2-cells $\epsilon$ as above display the functor $N_*$ as the right adjoint to the inclusion $ (\bb{N}_p X)_1  \subseteq (\bb{Z} \chi_p)_1$. 
    
\end{construction}

\begin{theorem}[Dugger and Spivak, \cite{dugger2011rigidification}]
    Let $X$ be a simplicial set and $a,b \in X_0$. A $p$-arrow, $[p] \in \D{}$, in the set $(\ff{C}X)_p(a,b)$ is represented by 
    \begin{itemize}
        \item [-] A necklace $|N| \rightarrow X_{a,b}$.
        \item[-] A flag $\ra{U} =  ( U^0 \subseteq \dots \subseteq U^p )$ of subsets of $V_N$ such that $J_N \subseteq U^0$.
    \end{itemize}
    This data is subject to the following relation: two pairs $(|N| \rightarrow X, \ra{U})$ and $(|M| \rightarrow X, \ra{V})$ are identified 
    in case there is a commutative triangle in $\sset_{*,*}$
\[\begin{tikzcd}
	{|N|} && {|M|} \\
	& {X_{a,b}}
	\arrow["\rho", from=1-1, to=1-3]
	\arrow[from=1-1, to=2-2]
	\arrow[from=1-3, to=2-2]
\end{tikzcd}\]
such that $\rho(\ra{U}) = \ra{V}$.
    
\end{theorem}

\begin{proof}
    By Theorem \ref{thm_colim}, a $p$-arrow $a \rightarrow b$ in $\ff{C}X$ is represented by a $\chi_p$-decorated zig-zag $(|T| \rightarrow X, \ra{U}) : (x,a) \rightsquigarrow (y,b)$. We may apply the necklace replacement functor defined in Construction \ref{const_ncsrep}, so without loss of generality we may assume that $T$ is a necklace and $\ra{U}$ is a proper decoration.

    The relation such data is subject to is recorded in the 2-cells of the double category $\bb{N}_p X$. By part (ii) of Proposition \ref{prop_translate}, this relation is exactly the one stated in the theorem.
\end{proof}

\appendix

\section{Zig-zag gymnastics}

\subsection{Unital and transposition 2-cells} \label{app_nice}

We describe a couple of maneuvers which are useful in working with the double categories $\Tilde{\bb{Z}} \cc{J}$ and $\Tilde{\bb{Z}} F$. 

Let us denote by
\[\begin{tikzcd}
	i & i
	\arrow["{u_n}", Rightarrow, no head, from=1-1, to=1-2]
\end{tikzcd}\]
a zig-zag in $\cc{Z} \cc{J}$ which is given by a constant functor $u_n : \cc{Z}^n \rightarrow \cc{J}$ at $i \in \cc{J}$. Sometimes we do not specify the length of the zig-zag, but we leave it as understood in context. 

The trivial functor $\cc{Z}^n \rightarrow \cc{Z}^0$ induces a 2-cell in $\bb{Z} \cc{J}$
\[\begin{tikzcd}
	i & i \\
	i & i
	\arrow[""{name=0, anchor=center, inner sep=0}, "{u_n}", Rightarrow, no head, from=1-1, to=1-2]
	\arrow["{id_i}", from=1-2, to=2-2]
	\arrow[""{name=1, anchor=center, inner sep=0}, "{u_0}"', Rightarrow, no head, from=2-1, to=2-2]
	\arrow["{id_i}"', from=1-1, to=2-1]
	\arrow[shorten <=4pt, shorten >=4pt, Rightarrow, from=0, to=1]
\end{tikzcd}\]
We may regard the above as a 2-cell in $\Tilde{\bb{Z}} \cc{J}$ of the form
\[\begin{tikzcd}
	i & i \\
	i & i \\
	i & i
	\arrow[""{name=0, anchor=center, inner sep=0}, Rightarrow, no head, from=2-1, to=2-2]
	\arrow["{id_i}", from=2-2, to=3-2]
	\arrow[""{name=1, anchor=center, inner sep=0}, "{u_0}"', Rightarrow, no head, from=3-1, to=3-2]
	\arrow["{id_i}"', from=2-1, to=3-1]
	\arrow["{id_i}", from=2-1, to=1-1]
	\arrow["{id_i}"', from=2-2, to=1-2]
	\arrow[""{name=2, anchor=center, inner sep=0}, "{u_n}", Rightarrow, no head, from=1-1, to=1-2]
	\arrow[shorten <=4pt, shorten >=4pt, Rightarrow, from=0, to=1]
	\arrow["id"', shorten <=4pt, shorten >=4pt, Rightarrow, from=0, to=2]
\end{tikzcd}\]
In conclusion, we have 2-cells in $\Tilde{\bb{Z}} \cc{J}$ of the form
\[\begin{tikzcd}
	i & i \\
	i & i
	\arrow[""{name=0, anchor=center, inner sep=0}, "{u_n}", Rightarrow, no head, from=1-1, to=1-2]
	\arrow[""{name=1, anchor=center, inner sep=0}, "{u_0}"', Rightarrow, no head, from=2-1, to=2-2]
	\arrow[Rightarrow, no head, from=1-1, to=2-1]
	\arrow[Rightarrow, no head, from=1-2, to=2-2]
	\arrow[shorten <=4pt, shorten >=4pt, Rightarrow, squiggly, from=0, to=1]
\end{tikzcd} \ \ \ , \ \ \ 
\begin{tikzcd}[column sep=scriptsize]
	i & i \\
	i & i
	\arrow[""{name=0, anchor=center, inner sep=0}, "{u_n}"', Rightarrow, no head, from=2-1, to=2-2]
	\arrow[""{name=1, anchor=center, inner sep=0}, "{u_0}", Rightarrow, no head, from=1-1, to=1-2]
	\arrow[Rightarrow, no head, from=2-1, to=1-1]
	\arrow[Rightarrow, no head, from=2-2, to=1-2]
	\arrow[shorten <=4pt, shorten >=4pt, Rightarrow, squiggly, from=1, to=0]
\end{tikzcd}
\]
where the cell on the right is obtained by symmetry.
In fact, the lengths of the vertical zig-zags may vary as we wish. We call 2-cells as above \textit{unital}. 

In $\Tilde{\bb{Z}} \cc{J}$, horizontal and vertical morphisms are the same: they are zig-zags in $\cc{J}$. Given a zig-zag $(\Bar{u}, \Bar{v}) : i \rightsquigarrow j$, there is a distinguished 2-cell of the form
\[\begin{tikzcd}
	i & j \\
	j & j
	\arrow[""{name=0, anchor=center, inner sep=0}, "{(\Bar{u}, \Bar{v})}", squiggly, from=1-1, to=1-2]
	\arrow["{(\Bar{u}, \Bar{v})}"', squiggly, from=1-1, to=2-1]
	\arrow[""{name=1, anchor=center, inner sep=0}, Rightarrow, no head, from=2-1, to=2-2]
	\arrow[Rightarrow, no head, from=1-2, to=2-2]
	\arrow[shorten <=4pt, shorten >=4pt, Rightarrow, squiggly, from=0, to=1]
\end{tikzcd}\]
which we describe as follows.

In case of a zig-zag of length $1$, say $i \xleftarrow{u} k \xrightarrow{v} j$, the above 2-cell is given by the following commutative diagram in $\cc{J}$
\[\begin{tikzcd}
	i & k & j \\
	k & k & j \\
	j & j & j
	\arrow["id"{description}, from=2-2, to=2-1]
	\arrow["id"{description}, from=2-2, to=1-2]
	\arrow["u", from=2-1, to=1-1]
	\arrow["u"', from=1-2, to=1-1]
	\arrow["v", from=1-2, to=1-3]
	\arrow["v"', from=2-1, to=3-1]
	\arrow["v"{description}, from=2-2, to=3-2]
	\arrow["id"{description}, from=2-3, to=3-3]
	\arrow["id"{description}, from=2-3, to=1-3]
	\arrow["v"{description}, from=2-2, to=2-3]
	\arrow["id", from=3-2, to=3-1]
	\arrow["id"', from=3-2, to=3-3]
\end{tikzcd}\]
For longer zig-zags, we can proceed recursively. 

Because of symmetry, we also have 2-cells of the form
\[\begin{tikzcd}
	i & i \\
	i & j
	\arrow[""{name=0, anchor=center, inner sep=0}, Rightarrow, no head, from=1-1, to=1-2]
	\arrow[Rightarrow, no head, from=1-1, to=2-1]
	\arrow[""{name=1, anchor=center, inner sep=0}, "{(\Bar{u}, \Bar{v})}"', squiggly, from=2-1, to=2-2]
	\arrow["{(\Bar{u}, \Bar{v})}", squiggly, from=1-2, to=2-2]
	\arrow[shorten <=4pt, shorten >=4pt, Rightarrow, squiggly, from=0, to=1]
\end{tikzcd}\]

We call these two types of cells \textit{transpositions}, since they transpose zig-zags from the vertical to the horizontal direction. In general, by applying unital and transposition 2-cells, any 2-cell in $\Tilde{\bb{Z}} \cc{J}$ 
\[\begin{tikzcd}
	i & j \\
	{i^\prime} & {j^\prime}
	\arrow[""{name=0, anchor=center, inner sep=0}, squiggly, from=1-1, to=1-2]
	\arrow[""{name=1, anchor=center, inner sep=0}, squiggly, from=2-1, to=2-2]
	\arrow[squiggly, from=1-1, to=2-1]
	\arrow[squiggly, from=1-2, to=2-2]
	\arrow[shorten <=4pt, shorten >=4pt, Rightarrow, squiggly, from=0, to=1]
\end{tikzcd}\]
may be transposed to a 2-cell of the form
\[\begin{tikzcd}
	i & j & {j^\prime} \\
	i & {i^\prime} & {j^\prime}
	\arrow[squiggly, from=2-2, to=2-3]
	\arrow[squiggly, from=1-2, to=1-3]
	\arrow[squiggly, from=1-1, to=1-2]
	\arrow[squiggly, from=2-1, to=2-2]
	\arrow[Rightarrow, no head, from=1-1, to=2-1]
	\arrow[Rightarrow, no head, from=1-3, to=2-3]
	\arrow[Rightarrow, squiggly, from=1-2, to=2-2]
\end{tikzcd}\]

In the double category $\Tilde{\bb{Z}} F$, we denote by 
\[\begin{tikzcd}
	{(i,a)} & {(i,a)}
	\arrow[Rightarrow, no head, from=1-1, to=1-2]
\end{tikzcd}\]
a decoration of a trivial zig-zag $\cc{Z}^n \rightarrow \cc{J}$ by identity morphisms $id_a$. 

In $\Tilde{\bb{Z}} F$, the vertical morphisms are $Ob(F)$-decorated, and hence we may regard them as horizontal morphisms decorated by identities. It is easy to check that the above unital and transposition 2-cells exist in $\Tilde{\bb{Z}} F$ as well. 

\subsection{Proof of Proposition \ref{prop_good}} \label{app_prop}

 \begin{lemma}
    The morphism $id_{[a]}$ is well-defined.
\end{lemma}

\begin{proof}
    Let $(i,a)$ and $(i^\prime, a^\prime)$ be two representatives for $[a]$. There exist an $ObF$-decorated zig-zag $(a,i) \rightsquigarrow (i^\prime, a^\prime)$, which we may regard as a 2-cell in $\Tilde{\bb{Z}}F$ and hence identifies $(\{ i\}, id_a)$ and $(\{ i^\prime\}, id_{a^\prime})$
\end{proof}

\begin{lemma}
    Any $ObF$-decorated zig-zag $(i,a) \rightsquigarrow (i^\prime, a^\prime)$ represents $id_{[a]}$.
\end{lemma}

\begin{proof}
    For a chosen $ObF$-decorated zig-zag, we have a 2-cell in $\title{\bb{Z}}F$ of the form
    \[\begin{tikzcd}
	{(i,a)} & {(i^\prime, a^\prime)} \\
	{(i^\prime, a^\prime)} & {(i^\prime, a^\prime)}
	\arrow[squiggly, from=1-1, to=2-1]
	\arrow[""{name=0, anchor=center, inner sep=0}, Rightarrow, no head, from=2-1, to=2-2]
	\arrow[""{name=1, anchor=center, inner sep=0}, squiggly, from=1-1, to=1-2]
	\arrow[Rightarrow, no head, from=1-2, to=2-2]
	\arrow[shorten <=4pt, shorten >=4pt, Rightarrow, squiggly, from=1, to=0]
\end{tikzcd}\]
The zig-zag on the bottom is evidently equivalent to $(\{ i^\prime \}, id_{a^\prime})$. 
\end{proof}

\begin{lemma} \label{lemma_1}
    A morphism $[a] \rightarrow [b]$ in $\cc{C}$ represented by a $F$-decorated zig-zag $(i,a) \rightsquigarrow (j,b)$, is also represented by the composite of the latter with any $ ObF$-decorated zig-zag $(j,b) \rightsquigarrow (j^\prime, b^\prime)$. 
\end{lemma}

\begin{proof}
    We have the following diagram in $\Tilde{\bb{Z}} F$
    \[\begin{tikzcd}[column sep=scriptsize]
	{(i,a)} && {(j,b)} \\
	{(i,a)} & {(j,b)} & {(j,b)} \\
	{(i,a)} & {(j,b)} & {(j^\prime, b^\prime)}
	\arrow[""{name=0, anchor=center, inner sep=0}, squiggly, from=3-1, to=3-2]
	\arrow[""{name=1, anchor=center, inner sep=0}, squiggly, from=3-2, to=3-3]
	\arrow[squiggly, from=2-3, to=3-3]
	\arrow[Rightarrow, no head, from=2-2, to=3-2]
	\arrow[""{name=2, anchor=center, inner sep=0}, Rightarrow, no head, from=2-2, to=2-3]
	\arrow[Rightarrow, no head, from=2-1, to=3-1]
	\arrow[""{name=3, anchor=center, inner sep=0}, squiggly, from=2-1, to=2-2]
	\arrow[""{name=4, anchor=center, inner sep=0}, squiggly, from=1-1, to=1-3]
	\arrow[Rightarrow, no head, from=1-1, to=2-1]
	\arrow[Rightarrow, no head, from=1-3, to=2-3]
	\arrow[shorten <=4pt, shorten >=4pt, Rightarrow, squiggly, from=2, to=1]
	\arrow[shorten <=4pt, shorten >=4pt, Rightarrow, squiggly, from=3, to=0]
	\arrow[shorten <=3pt, Rightarrow, squiggly, from=4, to=2-2]
\end{tikzcd}\]
\end{proof}

\begin{lemma} \label{lemma_2}
    Given 2-cells in $\title{\bb{Z}}(F)$
    \[\begin{tikzcd}[column sep=scriptsize]
	{(i,a)} & {(j,b)} \\
	{(i^\prime, a^\prime)} & {(j^\prime, b^\prime)}
	\arrow[squiggly, from=1-2, to=2-2]
	\arrow[squiggly, from=1-1, to=2-1]
	\arrow[""{name=0, anchor=center, inner sep=0}, squiggly, from=1-1, to=1-2]
	\arrow[""{name=1, anchor=center, inner sep=0}, squiggly, from=2-1, to=2-2]
	\arrow[shorten <=4pt, shorten >=4pt, Rightarrow, squiggly, from=0, to=1]
\end{tikzcd} \ \ , \ \ 
\begin{tikzcd}[column sep=scriptsize]
	{(j,b)} & {(c,k)} \\
	{(j^\prime, b^\prime)} & {(c^\prime, k^\prime)}
	\arrow[squiggly, from=1-2, to=2-2]
	\arrow[""{name=0, anchor=center, inner sep=0}, squiggly, from=1-1, to=1-2]
	\arrow[""{name=1, anchor=center, inner sep=0}, squiggly, from=2-1, to=2-2]
	\arrow[squiggly, from=1-1, to=2-1]
	\arrow[shorten <=4pt, shorten >=4pt, Rightarrow, squiggly, from=0, to=1]
\end{tikzcd}
\]
there is alsways a 2-cell
\[\begin{tikzcd}[column sep=scriptsize]
	{(i,a)} & {(j,b)} & {(c,k)} \\
	{(i^\prime, a^\prime)} & {(j^\prime, b^\prime)} & {(c^\prime, k^\prime)}
	\arrow[squiggly, from=1-1, to=2-1]
	\arrow[squiggly, from=1-1, to=1-2]
	\arrow[squiggly, from=2-1, to=2-2]
	\arrow[squiggly, from=1-3, to=2-3]
	\arrow[squiggly, from=1-2, to=1-3]
	\arrow[squiggly, from=2-2, to=2-3]
	\arrow[Rightarrow, squiggly, from=1-2, to=2-2]
\end{tikzcd}\]
whose top and bottom are the composites of the tops and bottoms of the given 2-cells.
\end{lemma}

\begin{proof}
    We can transpose the given 2-cells in the form 
    \[
    \begin{tikzcd}[column sep=scriptsize]
	{(i,a)} & {(j,b)} & {(j^\prime, b^\prime)} \\
	{(i,a)} & {(i^\prime, a^\prime)} & {(j^\prime, b^\prime)}
	\arrow[squiggly, from=2-2, to=2-3]
	\arrow[squiggly, from=2-1, to=2-2]
	\arrow[Rightarrow, no head, from=1-1, to=2-1]
	\arrow[squiggly, from=1-1, to=1-2]
	\arrow[squiggly, from=1-2, to=1-3]
	\arrow[Rightarrow, no head, from=1-3, to=2-3]
	\arrow[Rightarrow, squiggly, from=1-2, to=2-2]
\end{tikzcd} \ \ , \ \ 
    \begin{tikzcd}[column sep=scriptsize]
	{(j,b)} & {(c,k)} & {(c^\prime, k^\prime)} \\
	{(j,b)} & {(j^\prime, b^\prime)} & {(c^\prime, k^\prime)}
	\arrow[squiggly, from=2-2, to=2-3]
	\arrow[squiggly, from=2-1, to=2-2]
	\arrow[Rightarrow, no head, from=1-1, to=2-1]
	\arrow[squiggly, from=1-1, to=1-2]
	\arrow[squiggly, from=1-2, to=1-3]
	\arrow[Rightarrow, no head, from=1-3, to=2-3]
	\arrow[Rightarrow, squiggly, from=1-2, to=2-2]
\end{tikzcd}
\]
Moreover, by adjusting the lengths of the vertical identities, these 2-cells are composable. The desired conclusion follows from the previous lemma.
\end{proof}

\begin{lemma}
    Composition in $\cc{C}$ is well-defined.
\end{lemma}

\begin{proof}
    Let $[a] \rightarrow [b]$ and $[b] \rightarrow [c]$ be composable morphisms in $\cc{C}$. Assume the composite morphism is represented by two triples composites of decorated zig-zags
    \[\begin{tikzcd}
	{(i_1, a_1)} & {(j_1, b_1)} & {(j_1^\prime, b_1^\prime)} & {(k_1, c_1)}
	\arrow[squiggly, from=1-1, to=1-2]
	\arrow[squiggly, from=1-2, to=1-3]
	\arrow[squiggly, from=1-3, to=1-4]
\end{tikzcd}\]
and
\[\begin{tikzcd}
	{(i_2, a_2)} & {(j_2, b_2)} & {(j_2^\prime, b_2^\prime)} & {(k_2, c_2)}
	\arrow[squiggly, from=1-1, to=1-2]
	\arrow[squiggly, from=1-2, to=1-3]
	\arrow[squiggly, from=1-3, to=1-4]
\end{tikzcd}\]

We have a 2-cell 
\[\begin{tikzcd}
	{(i_1, a_1)} & {(j_1, b_1)} & {(j_1^\prime, b_1^\prime)} \\
	{(i_2, a_2)} & {(j_2, b_2)} & {(j_2^\prime, b_2^\prime)}
	\arrow[squiggly, from=1-1, to=1-2]
	\arrow[squiggly, from=1-2, to=1-3]
	\arrow[squiggly, from=2-1, to=2-2]
	\arrow[squiggly, from=2-2, to=2-3]
	\arrow[squiggly, from=1-1, to=2-1]
	\arrow[squiggly, from=1-3, to=2-3]
	\arrow[Rightarrow, squiggly, from=1-2, to=2-2]
\end{tikzcd}\]
since by Lemma \ref{lemma_1}, the top and bottom composites both represent $[a] \rightarrow [b]$. Similarly, we also have a 2-cell 
\[\begin{tikzcd}
	{(j_1^\prime, b_1^\prime)} & {(k_1, c_1)} \\
	{(j_2^\prime, b_2^\prime)} & {(k_2, c_2)}
	\arrow[""{name=0, anchor=center, inner sep=0}, squiggly, from=1-1, to=1-2]
	\arrow[""{name=1, anchor=center, inner sep=0}, squiggly, from=2-1, to=2-2]
	\arrow[squiggly, from=1-2, to=2-2]
	\arrow[squiggly, from=1-1, to=2-1]
	\arrow[shorten <=4pt, shorten >=4pt, Rightarrow, squiggly, from=0, to=1]
\end{tikzcd}\]

By Lemma \ref{lemma_2}, there exist a 2-cell
\[\begin{tikzcd}
	{(i_1, a_1)} & {(j_1, b_1)} & {(j_1^\prime, b_1^\prime)} & {(k_1, c_1)} \\
	{(i_2, a_2)} & {(j_2, b_2)} & {(j_2^\prime, b_2^\prime)} & {(k_2, c_2)}
	\arrow[squiggly, from=1-1, to=1-2]
	\arrow[""{name=0, anchor=center, inner sep=0}, squiggly, from=1-2, to=1-3]
	\arrow[squiggly, from=2-1, to=2-2]
	\arrow[""{name=1, anchor=center, inner sep=0}, squiggly, from=2-2, to=2-3]
	\arrow[squiggly, from=1-1, to=2-1]
	\arrow[squiggly, from=1-3, to=1-4]
	\arrow[squiggly, from=2-3, to=2-4]
	\arrow[squiggly, from=1-4, to=2-4]
	\arrow[shorten <=4pt, shorten >=4pt, Rightarrow, squiggly, from=0, to=1]
\end{tikzcd}\]
which identifies the triple composites used in expressing the composition in $\cc{C}$.
\end{proof}

It easily follows that composition in $\cc{C}$ is unital and associative. 

\printbibliography

\end{document}